\documentclass[12pt]{amsart}
\usepackage{amsthm,amssymb,amscd}
\usepackage{mathtools}
\usepackage{stmaryrd}
\usepackage[all]{xy}
\usepackage[numbers,sort&compress]{natbib}
\usepackage{bm}

\usepackage[%
dvipdfm,%
colorlinks=true,%
bookmarks=true,%
bookmarksnumbered=true,%
pdftitle={Moduli spaces of meromorphic connections and quiver varieties},%
pdfauthor={Kazuki Hiroe and Daisuke Yamakawa}%
]{hyperref}

\newtheorem{theorem}{Theorem}[section]
\newtheorem{corollary}[theorem]{Corollary}
\newtheorem{proposition}[theorem]{Proposition}
\newtheorem{lemma}[theorem]{Lemma}

\theoremstyle{definition}    
\newtheorem{definition}[theorem]{Definition}
\newtheorem{example}[theorem]{Example}
\newtheorem{remark}[theorem]{Remark}

\theoremstyle{remark}

\newenvironment{thm}{\begin{theorem}}{\end{theorem}}
\newenvironment{prp}{\begin{proposition}}{\end{proposition}}
\newenvironment{crl}{\begin{corollary}}{\end{corollary}}
\newenvironment{dfn}{\begin{definition}}{\end{definition}}
\newenvironment{lmm}{\begin{lemma}}{\end{lemma}}
\newenvironment{rem}{\begin{remark}}{\end{remark}}

\newcommand{\thmref}[1]{Theorem~\ref{#1}}
\newcommand{\secref}[1]{Section~\ref{#1}}
\newcommand{\lmmref}[1]{Lemma~\ref{#1}}
\newcommand{\prpref}[1]{Proposition~\ref{#1}}
\newcommand{\crlref}[1]{Corollary~\ref{#1}}
\newcommand{\dfnref}[1]{Definition~\ref{#1}}
\newcommand{\remref}[1]{Remark~\ref{#1}}

\newcommand{\Z}{\mathbb{Z}} 
\newcommand{\C}{\mathbb{C}} 
\newcommand{\bbP}{\mathbb{P}} 

\newcommand{\GL}{\operatorname{GL}} 
\newcommand{\gl}{\operatorname{\mathfrak{gl}}}
\newcommand{\frg}{\mathfrak{g}} 
\newcommand{\frh}{\mathfrak{h}} 
\newcommand{\frt}{\mathfrak{t}} 
\newcommand{\Lie}{\operatorname{Lie}} 
\newcommand{\Ad}{\operatorname{Ad}} 
\newcommand{\ad}{\operatorname{ad}}

\newcommand{\End}{\operatorname{End}}

\newcommand{\Hom}{\operatorname{Hom}}
\newcommand{\Ker}{\operatorname{Ker}}
\newcommand{\range}{\operatorname{Im}}

\newcommand{\tr}{\operatorname{tr}}

\newcommand{\prj}{\operatorname{pr}}

\newcommand{\Mat}{\operatorname{Mat}}

\newcommand{\GIT}{/\hspace{-3pt}/} 
\newcommand{\calO}{\mathcal{O}} 
\newcommand{\hcalO}{\widehat{\mathcal{O}}} 

\newcommand{\fp}[1]{\ensuremath{[\![#1]\!]} }
\newcommand{\cp}[1]{\ensuremath{\{#1\}} }
\newcommand{\fl}[1]{\ensuremath{(\!(#1)\!)} }
\newcommand{\res}{\operatorname*{res}}

\newcommand{\set}[2]%
{\left\{ \, #1 \, \middle| \, #2 \, \right\} }
\newcommand{\ov}{\overline}
\newcommand{\wh}{\widehat}
\newcommand{\wt}{\widetilde}
\newcommand{\sumfrac}[2]{\genfrac{}{}{0pt}{2}{#1}{#2}}

\newcommand{\sfQ}{\mathsf{Q}} 
\newcommand{\sft}{\mathsf{t}} 
\newcommand{\sfs}{\mathsf{s}} 
\newcommand{\Rep}{\operatorname{Rep}}
\newcommand{\bv}{\mathbf{v}}
\newcommand{\bw}{\mathbf{w}}

\newcommand{\calM}{\mathcal{M}}
\newcommand{\tcalM}{\widetilde{\mathcal{M}}}
\newcommand{\frM}{\mathfrak{M}}

\newcommand{\tO}{\widetilde{O}}
\newcommand{\frb}{\mathfrak{b}}
\newcommand{\frp}{\mathfrak{p}}
\newcommand{\fru}{\mathfrak{u}}
\newcommand{\calP}{\mathcal{P}}
\newcommand{\calU}{\mathcal{U}}
\newcommand{\frP}{\mathfrak{P}}
\newcommand{\frU}{\mathfrak{U}}
\newcommand{\sirr}{\mathrm{irr}}
\newcommand{\sres}{\mathrm{res}}

\newcommand{\calE}{\mathcal{E}}
\newcommand{\calK}{\mathcal{K}}
\newcommand{\hcalK}{\widehat{\mathcal{K}}}
\newcommand{\bfT}{\mathbf{T}}
\newcommand{\bfa}{\mathbf{a}}
\newcommand{\mx}{\mathfrak{m}}
\newcommand{\hmx}{\widehat{\mathfrak{m}}}
\newcommand{\Stab}{\operatorname{Stab}}

\newcommand{\bfL}{\mathbf{L}}
\newcommand{\bfO}{\mathbf{O}}
\newcommand{\tbO}{\widetilde{\mathbf{O}}}
\newcommand{\bfH}{\mathbf{H}}

\title[Moduli spaces of meromorphic connections and quiver varieties]%
{Moduli spaces of meromorphic connections and quiver varieties}

\author{Kazuki Hiroe}
\address{Department of Mathematics, Josai University, 
1-1 Keyakidai, Sakado City, Saitama 350-0295, Japan}
\email{kazuki@josai.ac.jp}

\author{Daisuke Yamakawa}
\address{Department of Mathematics,
Tokyo Institute of Technology,
2-12-1 Ookayama, Meguro-ku, Tokyo 152-8551, Japan}
\email{yamakawa@math.titech.ac.jp}
\thanks{The second named author is supported by 
JSPS Grant-in-Aid for Young Scientists (B) 
Grant Number 24740104.}


\begin{document}
\mathtoolsset{showonlyrefs=true}

\maketitle

\begin{abstract}
We describe the moduli spaces of meromorphic connections on 
trivial holomorphic vector bundles over the Riemann sphere 
with at most one (unramified) irregular singularity and 
arbitrary number of simple poles as Nakajima's quiver varieties.
This result enables us to solve partially 
the additive irregular Deligne-Simpson problem. 
\end{abstract}

\tableofcontents

\section{Introduction}

This paper is devoted to study the relationship between 
two families of complex symplectic manifolds:
one is certain moduli spaces of 
systems of linear ordinary differential equations 
with rational coefficients 
(i.e., meromorphic connections 
on trivial holomorphic vector bundles 
over the complex projective line $\bbP^1$),
and the other is 
Nakajima's quiver varieties~\cite{Nak94}
(with the real parameter taken to be zero and no framing).

\medskip

Take positive integers $k_1, k_2, \dots , k_m$, 
and for each $i=1,2, \dots , m$ 
let $O_i$ be a coadjoint orbit of the complex Lie group 
$G_{k_i} := \GL_n(\C[z_i]/(z_i^{k_i}))$,
where $z_i$ is an indeterminate.
Define 
\[
\calM^*_s = \set{(A_i) \in \prod_{i=1}^m O_i}%
{\text{$(A_i)$ is ``stable''},\ \sum_{i=1}^m \pi_\sres(A_i) =0}/\GL_n(\C),
\] 
where 
$\pi_\sres \colon (\Lie G_{k_i})^* \to \gl_n(\C)^*$ is the projection.
We do not give here the definition of ``stability'' (see \dfnref{dfn:stable}), 
which is some open condition to make the quotient a geometrically nice space.
Since the map $(A_i) \mapsto \sum_i \pi_\sres(A_i)$ is a moment map 
for the diagonal action of $\GL_n(\C)$ on $\prod_i O_i$, 
one can show that $\calM^*_s$ is a (smooth) complex symplectic manifold.

To view it as a certain moduli space of meromorphic connections on 
the trivial holomorphic vector bundle $\calO_{\bbP^1}^{\oplus n}$,
take distinct points $t_1, t_2, \dots , t_m \in \bbP^1$ 
and (for simplicity) a standard coordinate on $\bbP^1$ 
so that $z(t_i) \neq \infty$.
For each $i=1,2, \dots ,m$, identify each $z_i$ with a coordinate $z-z(t_i)$
and embeds $(\Lie G_{k_i})^*$ into $\gl_n(\C[z_i^{-1}]) z_i^{-1}dz_i$ 
using the residue and trace operations.
Then each $(A_i) \in \prod_i O_i$ gives a meromorphic connection 
$d-\sum_i A_i$ on $\calO_{\bbP^1}^{\oplus n}$ with poles at $t_i$'s;
it is holomorphic at $\infty$ if and only if $\sum_i \pi_\sres(A_i) =0$.

The space $\calM^*_s$ cannot be thought of 
as a ``moduli space of meromorphic connections''; 
we are not taking into account 
meromorphic connections on non-trivial holomorphic vector bundles.
However, it inherits many interesting structures 
from the moduli space of meromorphic connections.
A problem asking when $\calM^*_s$ is non-empty 
is called the {\em additive (irregular) Deligne-Simpson problem}.
It was solved by Crawley-Boevey~\cite{CB03DS} 
in the case where $k_i=1$ for all $i$ 
(namely, in the case of logarithmic connections)
and by Boalch~\cite{Boa12} in the case where 
one of $k_i$, say $k_1$, is less than $4$ 
and the others are equal to $1$,
and furthermore $O_1$ contains a(n unramified) ``normal form'' 
(see e.g. \cite[Definition~10]{Yam11} or Remark~\ref{rem:normal}).
Their approach is to describe $\calM^*_s$ as a quiver variety;
then the problem is immediately solved since  
we have a criterion~\cite{CB01} for the non-emptiness of quiver varieties.
We briefly review their results below.

\medskip

Let $\sfQ$ be a finite quiver (directed graph) 
with the set of vertices $\sfQ^v$.
To $(\bv, \zeta) \in \Z_{\geq 0}^{\sfQ^v} \times \C^{\sfQ^v}$,
one can associate a complex symplectic manifold $\frM_\sfQ^s(\bv,\zeta)$,
called the quiver variety.
This is some smooth open subset 
(the ``stable'' part) of the holomorphic symplectic quotient 
$T^* \Rep_\sfQ(V) \GIT_\zeta G_V$ at the level specified by $\zeta$,
where $\Rep_\sfQ(V)$ is the space of representations of $\sfQ$ 
over a collection of vector spaces $V=(V_i)_{i \in \sfQ^v}$ with dimension $\bv$, 
and $G_V := \prod_i \GL(V_i)$ (see \secref{subsec:quiver} for the precise definition).
They have rich geometric structures 
related to the gauge theory, singularity theory 
and representation theory 
of (symmetric) Kac-Moody algebras/quantum enveloping algebras 
of Drinfel'd-Jimbo~(see \cite{Nak94,Nak98,Nak01AMS,Nak01,Nak04} and references therein).

Crawley-Boevey~\cite{CB03DS} showed that if $k_i =1$ for all $i$, 
each $\calM^*_s$ is isomorphic to a quiver variety. 
One can easily check that his isomorphism 
intertwines the symplectic structures (see e.g.~\cite{Yam08}).
Furthermore, in this case 
the so-called reflection functors~\cite{CB01,Nak03} 
$S_i \colon \frM_\sfQ^s(\bv,\zeta) \xrightarrow{\simeq} 
\frM_\sfQ^s(s_i(\bv),s_i^T(\zeta))$, $i \in \sfQ^v$ 
of the quiver varieties 
(they satisfy the defining relation for the simple reflections $s_i$
generating the Weyl group associated to the quiver)
are expressed in terms of an ``additive'' analogue~\cite{DR00,Yam11} 
of Katz's middle convolutions~\cite{Katz,Ari10} 
and the tensor operation by rank one connections (see~\cite{Boa09,Yam10}).

Also, Boalch~\cite{Boa08,Boa12} showed that 
if $k_1 \leq 3$ and $k_i=1$ for $i \neq 1$,
and if $O_1$ contains a normal form,
then $\calM^*_s$ is symplectomorphic to a quiver variety.
In the proof he first described $\calM^*_s$ as 
(the stable part of) 
a holomorphic symplectic quotient 
of some larger symplectic manifold $\tcalM^*$,
called the {\em extended moduli space}.
This space is, roughly speaking, 
a certain moduli space of 
meromorphic connections equipped with a sort of framing at each pole,
and was originally introduced by himself~\cite{Boa01} to 
construct intrinsically 
the ``space of singularity data'' of 
Jimbo et al.~\cite{JMU},
the phase space for the (lifted) isomonodromic deformation of meromorphic connections.
He showed that in the above case there is an equivariant symplectomorphism 
$\tcalM^* \simeq (T^* \GL_n(\C))^{m-1} \times T^* \Rep_\sfQ(V)$ for some $\sfQ$ and $V$, 
and that it induces a symplectomorphism between $\calM^*_s$ and a quiver variety. 
He also constructed Weyl group symmetries of families of spaces $\calM^*_s$ 
whose parameter behavior coincides with that of $(\bv,\zeta)$ 
under the reflection functors through his isomorphism.

In \cite[Appendix~C]{Boa08} he further ``claimed'' that 
there is still an equivariant symplectomorphism  
$\tcalM^* \simeq (T^* \GL_n(\C))^{m-1} \times T^* \Rep_\sfQ(V)$ for some $\sfQ$ and $V$
when one drops the assumption $k_1 \leq 3$,
and explained how to construct $\sfQ$ and $V$.
However it seems that 
only the existence of an equivariant isomorphism (not symplectomorphism) 
was shown there.

\medskip

In this paper we justify his claim and 
show that $\calM^*_s$ is symplectomorphic to a quiver variety
if $k_i = 1$ for $i \neq 1$ and $O_1$ contains a normal form.
Thanks to \cite{CB01}, our result immediately gives,
under the same assumption, 
an answer to the additive irregular Deligne-Simpson problem,
in terms of the root system attached to the quiver. 


After we wrote the first version of this paper, 
the first named author~\cite{Hir} used the methods developed here to
solve the additive irregular Deligne-Simpson problem for the case 
where each $O_i$ contains a normal form (with $k_i$ arbitrary),
although in this case the space $\calM^*_s$ may not be a quiver variety
(see~\cite[p.~3]{Boa08}).

The organization of this paper is as follows.
In \secref{sec:pre}, we introduce Boalch's extended moduli spaces
and review their basic properties.
\secref{sec:str-ext-moduli} is devoted to prove 
our main result stated in \secref{subsec:quiver-moduli}.
Also, we put an appendix (\secref{sec:app}) consisting of 
some basic facts on the classical formal reduction theory 
of meromorphic connections 
and on coadjoint orbits of general linear groups,
which will be used in 
Sections \ref{sec:pre} and \ref{sec:str-ext-moduli}.


\section{Preliminary notions and facts}\label{sec:pre}

In this section we define the extended moduli spaces 
and show their basic properties following Boalch~\cite{Boa01}.
Strictly speaking, our definition is a slight generalization of 
that in \cite{Boa01}, where the irregular types (see below) are assumed to 
have regular semisimple top coefficient.
Although the generalization is straightforward and already known for specialists,
we decided to add the contents of this section to our paper 
for the reader's convenience.

Throughout this section we fix a non-empty finite subset $D$ of $\bbP^1$. 
Let $\calO \equiv \calO_{\bbP^1}$ (resp.\ $\Omega^1 \equiv \Omega^1_{\bbP^1}$) be 
the sheaf of holomorphic functions
(resp.\ one-forms) on $\bbP^1$ and
$\Omega^1(*D)$ (resp.\ $\Omega^1(\log D)$)
the sheaf of meromorphic (resp.\ logarithmic) one-forms
on $\bbP^1$ with poles on $D$ (and no pole elsewhere).
For $t \in \bbP^1$, let $(\hcalO_t,\hmx_t)$ be 
the formal completion of the local ring $(\calO_t,\mx_t)$
and $\calK_t$ (resp.\ $\hcalK_t$) the field of fractions
of $\calO_t$ (resp.\ $\hcalO_t$).
For $\calO_t$-module $M$, we set $\wh{M} = \hcalO_t \otimes_{\calO_t} M$.

\subsection{Extended moduli spaces}

Let $G$ be a (connected) complex reductive group 
and $\frg = \Lie G$ its Lie algebra.
Fix a maximal torus $\frt$ of $\frg$. 
An (unramified) {\em irregular type} at $t \in \bbP^1$ is an element of 
\[
\frt(\hcalK_t)/\frt(\hcalO_t) = (\frt \otimes_\C \hcalK_t)/( \frt \otimes_\C \hcalO_t).
\]
In terms of a local coordinate $z$ vanishing at $t$,
an irregular type may be regarded as an element of
\[
\frt\fl{z}/\frt\fp{z} \simeq z^{-1}\frt[z^{-1}].
\]

Let $\calE \to \bbP^1$ be a holomorphic principal $G$-bundle on $\bbP^1$.
By a {\em meromorphic connection} on $\calE$ with poles on $D$
we mean a $\frg$-valued meromorphic one-form $A$ on $\calE$ 
with poles on $\calE |_D$ satisfying 
\[
\Ad_a (R_a^* A) = A \quad (a \in G),
\qquad
A(X^*) = -X \quad (X \in \frg),
\]
where $R_a$ is the canonical right action on $\calE$ by $a$ 
and $X^*$ is the fundamental vector field associated to $X$ 
(i.e., $X^*_p = \partial_t (p \cdot e^{tX}) |_{t=0}$ for $p \in \calE$).
Beware that our sign convention for the second condition
is different to the standard one.

\begin{dfn}
Let $\calE \to \bbP^1$ be a holomorphic principal $G$-bundle on $\bbP^1$,
$A$ a meromorphic connection on $\calE$ with poles on $D$
and $\bfa =(a_t)_{t \in D} \in \prod_{t \in D} \calE_t$ 
a collection of points of the fibers $\calE_t$, $t \in D$ of $\calE$.
The triple $(\calE,A,\bfa)$ is called 
a {\em compatibly framed meromorphic connection} on $(\bbP^1,D)$ 
if for each $t \in D$, there exist 
a germ $\wt{a}_t$ at $t$ of local sections of $\calE$ through $a_t$
and an irregular type $T_t \in \frt(\hcalK_t)/\frt(\hcalO_t)$
such that 
\begin{equation}
\wt{a}_t^*(A) - dT_t \in \frg \otimes_\C \Omega^1(\log t)_t.
\label{eq:irregular}
\end{equation}
\end{dfn}

It is known that the above $T_t$ is unique (see \remref{rem:irr-unique});
so it is called {\em the} irregular type of $(\calE,A,\bfa)$ at $t$.

\begin{rem}
(i) Any local coordinate $z$ vanishing at $t$ induces an identification
\[
\frg \otimes_\C \Omega^1(\log t)_t
\simeq z^{-1}\frg\cp{z}dz.
\]
Therefore in terms of the coordinate $z$, 
condition~\eqref{eq:irregular} 
means that $z(\wt{a}_t^*(A)-dT_t)$ has no terms of negative degree.

(ii) If $\calE = \bbP^1 \times G$, 
then pulling back via the trivial section $\bbP^1 \to \calE$, $x \mapsto (x,1)$
enables us to regard $A$ 
as a $\frg$-valued one-form on $\bbP^1$ with poles on $D$
and each $a_t$ as an element of $G$.
In this case, $(\calE,A,\bfa)$ is compatibly framed if and only if  
each $a_t$ admits a lift $\wt{a}_t \in G(\calO_t)$ such that 
\[
\wt{a}_t^{-1}[A] - dT_t \in \frg \otimes_\C \Omega^1(\log t)_t,
\]
where $\wt{a}_t^{-1}[A]$ is the gauge transform of $A$ 
(expressed as $\wt{a}_t^{-1} A \wt{a}_t - \wt{a}_t^{-1} d\wt{a}_t$ in any representation).
\end{rem}

\begin{dfn}
Two compatibly framed meromorphic connections 
$(\calE,A,\bfa)$ and $(\calE',A',\bfa')$ on $(\bbP^1,D)$ 
are said to be {\em isomorphic}
if there exists an isomorphism 
$\varphi \colon \calE \simeq \calE'$
such that $\varphi^*A' = A$ and $\varphi(a_t)=a'_t$ for each $t \in D$. 
\end{dfn}

If $(\calE,A,\bfa)$ and $(\calE',A',\bfa')$ are isomorphic,
then they have common irregular type at each $t \in D$.

Let $\bfT =(T_t)_{t \in D}$ be a collection of irregular types 
$T_t \in \frt(\hcalK_t)/\frt(\hcalO_t)$.
\begin{dfn}
The {\em extended moduli space} $\tcalM^*(\bfT)$ is 
the set of isomorphism classes of 
compatibly framed meromorphic connections $(\calE,A,\bfa)$
with irregular type $T_t$ at each $t \in D$
such that $\calE \simeq \bbP^1 \times G$.
\end{dfn}

\subsection{Extended orbits}\label{subsec:ext-orbit}

The extended moduli space $\tcalM^*(\bfT)$ 
is known to be expressed as a (holomorphic)
symplectic quotient of the product of 
some complex symplectic manifolds,
called ``extended orbits''.

Fix an irregular type $T \in \frt(\hcalK_t)/\frt(\hcalO_t)$
at some $t \in \bbP^1$.
Let $k>1$ be the pole order of $dT$ when $T \neq 0$, and $k:=1$ when $T=0$.
Set 
\[
G_k = G(\hcalO_t/\hmx_t^k) = G(\calO_t/\mx_t^k), \quad
\frg_k = \Lie G_k = \frg(\hcalO_t/\hmx_t^k),
\]
and let $G(\hcalO_t)_k$ be the kernel of the group homomorphism 
$G(\hcalO_t) \to G_k$ induced from 
the truncation $\hcalO_t \to \hcalO_t/\hmx_t^k$; 
so $G_k = G(\hcalO_t)/G(\hcalO_t)_k$.
Also set 
\[
B_k = G(\hcalO_t)_1/G(\hcalO_t)_k, \quad
\frb_k = \Lie B_k.
\]
The natural semidirect product decomposition 
$G_k = B_k \rtimes G$ induces 
direct sum decompositions
$\frg_k = \frb_k \oplus \frg$
and $\frg_k^* = \frb_k^* \oplus \frg^*$.
We denote by $\pi_\sirr \colon \frg_k^* \to \frb_k^*$ and 
$\pi_\sres \colon \frg_k^* \to \frg^*$ the projections.
The pairing
\begin{equation}\label{eq:pair}
\frg(\hcalK_t) \otimes_\C \left( \frg \otimes_\C \Omega^1(*t)^\wedge_t \right) \to \C;
\quad X \otimes A \mapsto (X,A) := \res_t \tr XA
\end{equation}
induces embeddings
\begin{align*}
\frg_k^* &\hookrightarrow 
\left( \frg \otimes_\C \Omega^1(*t)^\wedge_t \right) /
\left( \frg \otimes_\C \wh{\Omega}^1_t \right) 
= \frg \otimes_\C (\Omega^1(*t)^\wedge_t/\wh{\Omega}^1_t), \\
\frb_k^* &\hookrightarrow 
\left( \frg \otimes_\C \Omega^1(*t)^\wedge_t \right) /
\left( \frg \otimes_\C \Omega^1(\log t)^\wedge_t \right)
= \frg \otimes_\C (\Omega^1(*t)^\wedge_t/\Omega^1(\log t)^\wedge_t),
\end{align*}
by which we may regard $dT$ as an element of $\frb_k^*$.
Note that in terms of a local coordinate $z$ vanishing at $t$,
the above embeddings are respectively 
expressed as $\frg_k^* \hookrightarrow \frg[z^{-1}]dz/z$, $\frb_k^* \hookrightarrow \frg[z^{-1}]dz/z^2$.

\begin{dfn}
Let $O_B$ be the $B_k$-coadjoint orbit through $dT$.
Define
\[
\tO = \{\, (a,A) \in G \times \frg_k^* \mid \pi_\sirr(\Ad_a A) \in O_B\,\},
\]
which we call the {\em extended orbit} associated to $T$.
\end{dfn}

The following proposition gives a structure of complex symplectic manifold on 
the extended orbit $\tO$.

\begin{prp}[cf.~{\cite[Lemmas 2.2 and 2.4]{Boa01}}]\label{prop:decouple}
There are canonical bijections:
\[
\tO \simeq T^* G_k \times O_B \GIT B_k \simeq T^* G \times O_B,
\]
where in the middle term $B_k$ acts by the coadjoint action on $O_B$
and by the standard action on $T^* G_k$ 
coming from the left multiplication.
Furthermore, the last bijection is a symplectomorphism.
\end{prp}

\begin{proof}
A moment map for the $B_k$-action on $T^* G_k \times O_B$ is given by 
\[
\mu \colon T^*G_k \times O_B \to \frb_k^*; \quad
(g,A,B) \mapsto -\pi_\sirr(\Ad^*_g A) +B,
\]
where $T^* G_k \simeq G_k \times \frg_k^*$ 
via the left trivialization.
Define
\begin{equation}\label{eq:chi}
\chi \colon \mu^{-1}(0) \to G \times \frg_k^*; \quad
(g,A,B) \mapsto (g(t),A).
\end{equation}
It is straightforward to check that $\chi$ induces a bijection 
$T^* G_k \times O_B \GIT B_k \xrightarrow{\simeq} \tO$.
A bijection $\tO \xrightarrow{\simeq} T^* G \times O_B$ 
and its inverse are respectively given by
\[
(a,A) \mapsto (a,\pi_\sres(A), \pi_\sirr(\Ad_a A)), 
\quad
(a,R,B) \mapsto (a,\Ad_a^{-1}(B)+R).
\]
Through this bijection, the following map gives a section of $\chi$:
\[
T^*G \times O_B \to T^*G_k \times O_B;\quad
(a,R,B) \mapsto (a,\Ad_a^{-1}(B)+R,B),
\]
which is easily seen to be symplectic.
\end{proof}

It is easy to see that the $G$-action on $\tO$ defined by 
$f \cdot (a,A) = (af^{-1},\Ad_f A)$
is Hamiltonian with moment map
\[
\mu_G \colon \tO \to \frg^*; \quad \mu_G (a,A) = \pi_\sres (A) \in \frg^*.
\]

\begin{prp}[cf.~{\cite[Proposition 2.1]{Boa01}}]\label{prop:ext-moduli}
Let $\bfT =(T_t)_{t \in D}$ be a collection of 
irregular types $T_t \in \frt(\hcalK_t)/\frt(\hcalO_t)$.
For $t \in D$ let $\tO_t$ be the extended orbit associated to $T_t$ 
and set $\tbO = \prod_{t \in D} \tO_t$.
Then there exists a canonical bijection between 
the extended moduli space $\tcalM^*(\bfT)$ and 
the symplectic quotient $\tbO \GIT G$
of $\tbO$ by the action of $G$.
\end{prp}

\begin{proof}
Suppose that an element of $\tcalM^*(\bfT)$ is given.
Take a representative $(\calE,A,\bfa)$ of it so that $\calE = \bbP^1 \times G$. 
Then pulling back via the trivial section $\bbP^1 \to \calE$, $x \mapsto (x,1)$,
we may regard $A$ as a $\frg$-valued one-form on $\bbP^1$ with poles on $D$.
By the definition, for each $t \in D$, 
the compatible framing $a_t$ is an element of $G$ 
and admits a lift $\wt{a}_t \in G(\calO_t)$ such that 
\[
\wt{a}_t^{-1}[A] - dT_t \in \frg \otimes_\C \Omega^1(\log t)_t.
\]
Let $A_t$ be the element of
\[
\left( \frg \otimes_\C \Omega^1(*t)_t \right) / 
\left( \frg \otimes_\C \Omega^1_t \right) 
= \frg \otimes_\C (\Omega^1(*t)_t / \Omega^1_t)
= \frg \otimes_\C (\Omega^1(*t)^\wedge_t / \wh{\Omega}^1_t)
\]
represented by the germ of $A$ at $t$.
The above condition for $\wt{a}_t$ now implies that the pair $(a_t^{-1},A_t)$ 
is contained in $\tO_t$;
thus we obtain an element of $\tbO = \prod_{t \in D} \tO_t$.
The condition that $A$ is holomorphic away from $D$ 
is expressed as
\[
\sum_{t \in D} \res_t A_t =0,
\]
which is exactly the moment map condition $\sum_{t \in D} \mu_G(a_t^{-1},A_t) =0$.
Changing the choice of representative $(\calE,A,\bfa)$ 
corresponds to the simultaneous action of $G$ on $\tbO$;
so we may define a map $\tcalM^*(\bfT) \to \tbO \GIT G$.
The bijectivity of it immediately follows from 
the standard fact  
that taking principal part at each pole gives a bijection
\[
\Gamma(\bbP^1,\Omega^1(*D)) \simeq 
\set{(\alpha_t)_{t \in D} \in \prod_{t \in D} (\Omega^1(*t)_t /\Omega^1_t)}
{\sum_{t \in D} \res_t \alpha_t =0}.
\]
\end{proof}

In what follows, we identify the extended moduli space $\tcalM^*(\bfT)$
with the symplectic quotient $\tbO \GIT G$ 
via the bijection given above.

\subsection{Extended orbits and $G_k$-coadjoint orbits}\label{subsec:str-ext-orbit}

In this subsection we will give some relationship 
between extended orbits and $G_k$-coadjoint orbits.

Fix an irregular type $T \in \frt(\hcalK_t)/\frt(\hcalO_t)$ 
at some $t \in \bbP^1$ and let $k$ be as in \secref{subsec:ext-orbit}.
Set $H = \{\, h \in G \mid \Ad_h T = T \,\}$ and let $\frh$ be its Lie algebra.
The pairing \eqref{eq:pair} induces a non-degenerate pairing
\[
\frh \otimes_\C \left( \frh \otimes_\C (\Omega^1(\log t)_t^\wedge / \wh{\Omega}^1_t) \right)
\to \C.
\]
In what follows, 
we identify $\frh^*$ with $\frh \otimes_\C (\Omega^1(\log t)_t^\wedge / \wh{\Omega}^1_t)$ 
using this pairing.

Let $H$ act on the extended orbit $\tO$ associated to $T$ by $h \cdot (a,A)=(h a,A)$.
This is well-defined because if $\Ad^*_b \pi_\sirr(a A a^{-1}) = dT$
for some $b \in B_k$, then $hbh^{-1} \in B_k$ and 
\[
\Ad^*_{hbh^{-1}} \pi_\sirr(\Ad_{ha} A) = \Ad_h (dT) = dT.
\]
For $g \in G(\hcalK_t)$ and $A \in \frg \otimes_\C \Omega^1(*t)^\wedge_t$, 
let $g[A] = g A g^{-1} + dg \cdot g^{-1}$.

\begin{prp}\label{prop:reduction}
Let $A \in \frg \otimes_\C \Omega^1(*t)^\wedge_t$ 
and suppose 
$A - dT \in \frg \otimes_\C \Omega^1(\log t)^\wedge_t$.
Then there exists $\wh{b} \in G(\hcalO_t)_1$ such that 
\[
\wh{b}[A] - dT \in \frh \otimes_\C \Omega^1(\log t)^\wedge_t.
\]
Furthermore, $\res_t(\wh{b}[A]) \in \frh$ does not depend on $\wh{b}$. 
\end{prp}

To prove this we need some basic fact  
on the classical formal reduction theory 
of meromorphic connections,
which will be shown in \secref{subsec:reduction} (\crlref{cor:reduction}), 
and the following lemma:
\begin{lmm}\label{lem:exponent}
Assume $k>1$ and let 
$T' \in \frt(\hcalK_t)/\frt(\hcalO_t)$ be an irregular type at $t$ 
of pole order at most $k-1$.
Suppose that $g \in G_k$ and $L, L' \in \frh^*$ satisfy
\[
\Ad^*_g(dT + L)=dT' + L' \in \frg_k^*.
\]
Then the value $g(t) \in G$ of $g$ at $t$ satisfies 
$\Ad_{g(t)}T=T'$, $\Ad_{g(t)}L =L'$.
\end{lmm}

\begin{proof}
Replacing $g$, $T'$, $L'$ with $g(t)^{-1}g$, $\Ad_{g(t)}^{-1}T'$, $\Ad_{g(t)}^{-1}L'$,
respectively, we may assume $g(t)=1$.
Also, taking a faithful representation of $G$ if necessary, 
we may assume $G=\GL(n,\C)$.
Fix a local coordinate $z$ vanishing at $t$
and write
\[
g = \sum_{i=0}^{k-1} g_i z^i,
\quad
T = \sum_{i=1}^{k-1} T_i z^{-i},
\quad
T' = \sum_{i=1}^{k-1} T'_i z^{-i}.
\]
Set 
\begin{align*}
\frh_i &= \bigcap_{j=i+1}^{k-1} \Ker \ad_{T_j},
\quad (i=0,1, \dots ,k-2),
\qquad
\frh_{k-1}=\frg,
\\
\frh'_i &= \range \left( \ad_{T_{i+1}} |_{\frh_{i+1}} \right)
\quad (i=0,1, \dots ,k-2).
\end{align*}
The assumption is expressed as
\begin{align}
\sum_{j=0}^{k-1-i} (j+i) (g_j T_{j+i} - T'_{j+i} g_j) &=0 
\quad (i=1,\dots ,k-1),
\label{eq:stabilizer1}
\\
g_0 \res_t (L) - \res_t (L') g_0 &= \sum_{j=1}^{k-1} j (g_j T_j - T'_j g_j).
\label{eq:stabilizer2}
\end{align}
Note that $g_0=g(t)=1$.
Therefore equality \eqref{eq:stabilizer1} for $i=k-1$ implies $T_{k-1}=T'_{k-1}$.
If $k>2$, equality \eqref{eq:stabilizer1} for $i=k-2$ reads 
\[
(k-2)(T_{k-2} - T'_{k-2}) + (k-1)[g_1,T_{k-1}] =0.
\] 
Observe that the first term on the left hand side 
lies in $\Ker \ad_{T_{k-1}}=\frh_{k-2}$,
while the second term is contained in $\range \ad_{T_{k-1}}=\frh'_{k-2}$.
Since we have a decomposition 
\[
\frg = \frh_{k-2} \oplus \frh'_{k-2},
\]
we see that the both terms are zero;
hence $T_{k-2}=T'_{k-2}$ and $g_1 \in \frh_{k-2}$.
If $k>3$, look at \eqref{eq:stabilizer1} for $i=k-3$ and 
use the decomposition $\frg = \frh_{k-3} \oplus \frh'_{k-3} \oplus \frh'_{k-2}$;
then we obtain $T_{k-3}=T'_{k-3}$, $g_1 \in \frh_{k-3}$ and $g_2 \in \frh_{k-2}$.
Iterating this argument,
we finally obtain
\[
T=T', \qquad g_i \in \frh_i \quad (i=1,2, \dots ,k-2).
\]
Now look at \eqref{eq:stabilizer2};
the left hand side is contained in $\frh = \frh_0$ 
and each $[g_j,T_j]$ on the right hand side is contained in $\frh'_{j-1}$.
Using the decomposition
\[
\frg = \frh_0 \oplus \bigoplus_{i=0}^{k-2} \frh'_i,
\]
we obtain 
$g_i \in \frh_{i-1}\, (i=1, \dots ,k-2)$ and $L = L'$.
\end{proof}

\begin{rem}\label{rem:irr-unique}
The above in particular shows that 
the irregular type of any compatibly framed meromorphic connection 
is unique at each pole. 
\end{rem}

\begin{proof}[Proof of \prpref{prop:reduction}]
Take a local coordinate $z$ vanishing at $t$.
Then \crlref{cor:reduction} provides a desired $\wh{b}$.
The rest assertion follows from \lmmref{lem:exponent}.
\end{proof}

\begin{crl}\label{cor:exponent}
For any $(a,A) \in \tO$, 
there exists a unique $L \in \frh^*$ such that
\[
\Ad^*_{ba}(A) = dT + L \in \frg_k^*
\]
for some $b \in B_k$.
\end{crl}

We define a map $\mu_H \colon \tO \to \frh^*$ by 
$\mu_H(a,A) = - L$,
where $L \in \frh^*$ is given above.

\begin{prp}[cf.~{\cite[Lemma 2.3]{Boa01}}]\label{prop:exponent}
{\rm (i)} The map $\mu_H$ is a moment map for the $H$-action on $\tO$.

{\rm (ii)} For $L \in \frh^*$,
the symplectic quotient $\tO \GIT_{O(-L)} H$ of $\tO$ 
along the coadjoint orbit $O(-L)$ through $-L$ by the $H$-action 
is naturally isomorphic to the $G_k$-coadjoint orbit $O$ through $dT+L$.
\end{prp}

\begin{proof}
(i) Define an injective map 
$\iota \colon G_k \times \frh^* \to \mu^{-1}(0) \subset T^* G_k \times O_B$ 
by 
\[
\iota(g,R)=(g,\Ad^*_{g^{-1}}(dT+R),dT).
\]
Then \crlref{cor:exponent} implies that the composite 
$\chi \circ \iota \colon G_k \times \frh^* \to \tO$,
where $\chi$ is defined in \eqref{eq:chi}, is surjective.
Let $\prj \colon \mu^{-1}(0) \to T^* G_k$ be the first projection.
Since the $O_B$-component of $\iota$ is constant,
the pull-back of the symplectic structure on $T^* G_k$
along $\prj \circ \iota$ coincides with the pull-back of the
symplectic structure on $\tO$ along $\chi \circ \iota$.
Let $H$ act on $T^* G_k$ by the standard action
coming from the left multiplication,
on $O_B$ by conjugation and on $G_k \times \frh^*$
by $h \cdot (g,R)=(hg,R)$.
Then $\mu^{-1}(0)$ is $H$-invariant and 
$\chi, \iota, \prj$ are all $H$-equivariant.
A moment map on $T^* G_k$ is given by
\[
\nu \colon T^*G_k \to \frh^*; \quad
(g,A) \mapsto -\delta(\pi_\sres(\Ad^*_g A)),
\]
where $\delta \colon \frg^* \to \frh^*$ is the projection.
The statement now follows from the fact that
the pull-back of $\nu$ along $\prj \circ \iota$
is the pull-back of $\mu_H$ along $\chi \circ \iota$.

(ii) We have a $G_k$-action on $\tO$ given by
$g \cdot (a,A) = (a g(t)^{-1}, \Ad^*_g(A))$.
This action clearly commutes with the $H$-action and 
the second projection $(a,A) \mapsto A \in \frg_k^*$ gives a moment map.
Therefore it is sufficient to check that each fiber of the projection 
\[
\mu_H^{-1}(-L) \to O;\quad (a,A) \mapsto A
\]
is exactly a single orbit of the stabilizer $\Stab_H(L) \subset H$ of $L$. 
So assume $(a, A), (a', A) \in \mu_H^{-1}(-L)$.
Then there exist $b, b' \in B_k$ such that
\[
\Ad^*_{ba}(A) = \Ad^*_{b'a'}(A) = dT + L \in \frg_k^*.
\]
Letting $g=b'a'a^{-1}b^{-1} \in G_k$, we then obtain
\[
\Ad^*_g(dT + L)=dT + L \in \frg_k^*.
\]
\lmmref{lem:exponent} shows $a'a^{-1}=g(t) \in \Stab_H(L)$. 
\end{proof}

\begin{rem}\label{rem:normal}
The above $dT + L$ for each $L \in \frh$ is 
what we call a {\em normal form} in Introduction. 
\end{rem}

\begin{crl}\label{cor:moduli1}
Let $\bfT$, $\tbO$ be as in \prpref{prop:ext-moduli}.
Set 
\[
H_t = \{\, h \in G \mid \Ad_h T_t = T_t\,\}, \quad \frh_t = \Lie H_t
\quad (t \in D)
\]
and $\bfH = \prod_{t \in D} H_t$.
Take arbitrary $\bfL=(L_t)_{t \in D} \in \bigoplus_{t \in D} \frh_t^*$
and for $t \in D$, let $O_t$ be the $G_{k_t}$-coadjoint orbit 
through $dT_t + L_t \in \frg_{k_t}^*$
(where $k_t$ is the pole order of $dT_t$ when $T_t \neq 0$ and $k_t=1$ otherwise). 
Then there exists a canonical bijection
\[
\tcalM^*(\bfT) \GIT_{O(-\bfL)} \bfH 
\simeq \bfO \GIT G
\]
between the symplectic quotient of $\tcalM^*(\bfT)$  
along the $\bfH$-coadjoint orbit $O(-\bfL)$ through $-\bfL$ by the $\bfH$-action 
and that of the product $\bfO := \prod_{t \in D} O_t$ by the $G$-action.
\end{crl}

Note that the above $\bfO \GIT G$ may be singular.
The following immediately follows 
from the arguments in the proof of \prpref{prop:ext-moduli} 
and the above corollary:
\begin{crl}[cf.~{\cite[Proposition 2.1]{Boa01}}]\label{cor:moduli2}
Let $\bfT$, $\bfL$, $\bfO$ be as above.
Let $\calM^*(\bfT, \bfL)$ be the set of isomorphism classes of 
meromorphic connections $A \in \frg \otimes_\C \Omega^1(*D)$ 
on the trivial principal $G$-bundle $\bbP^1 \times G$ 
such that for each $t \in D$ there exists $g \in G(\calO_t)$ satisfying
\[
g[A] - dT_t - L_t \in \frg \otimes_\C \Omega^1_t.
\] 
Then there exists a canonical bijection between 
$\calM^*(\bfT,\bfL)$ and the symplectic quotient 
$\bfO \GIT G$.
\end{crl}

\section{Structure of (extended) moduli spaces}\label{sec:str-ext-moduli}

Throughout this section we assume $G=\GL(n,\C)$.

\subsection{Quivers and quiver varieties}\label{subsec:quiver}

Recall that a (finite) {\em quiver} is a quadruple 
$\sfQ=(\sfQ^v,\sfQ^a,\sfs,\sft)$ consisting of two finite sets 
$\sfQ^v, \sfQ^a$ (the vertices and arrows) and two maps 
$\sfs, \sft \colon \sfQ^a \to \sfQ^v$ (the source and target maps).
A {\em representation} (over $\C$) of $\sfQ$ is a pair $(V,\rho)$ 
consisting of a collection $V=(V_i)_{i \in \sfQ^v}$ of 
$\C$-vector spaces and a collection $\rho=(\rho_\alpha)_{\alpha \in \sfQ^a}$
of linear maps $\rho_\alpha \colon V_{\sfs(\alpha)} \to V_{\sft(\alpha)}$.
A {\em morphism} $\varphi \colon (V,\rho) \to (V',\rho')$ 
between two representations is a collection $\varphi=(\varphi_i)_{i \in \sfQ^v}$
of linear maps $\varphi_i \colon V_i \to V'_i$ such that 
$\varphi_{\sft(\alpha)} \circ \rho_\alpha = \rho_\alpha \circ \varphi_{\sfs(\alpha)}$
for all $\alpha \in \sfQ^a$.
Representations of $\sfQ$ form an abelian category.

Let $V=(V_i)_{i \in \sfQ^v}$ be a non-zero collection of 
finite-dimensional $\C$-vector spaces.
We define 
\[
\Rep_\sfQ(V) = \bigoplus_{\alpha \in \sfQ^a} \Hom(V_{\sfs(\alpha)},V_{\sft(\alpha)}),
\]
which is the space of representations $(V,\rho)$ with fixed $V$.
The group $G_V := \prod_{i \in \sfQ^v} \GL(V_i)$ acts on it as 
\[
g=(g_i)_{i \in \sfQ^v} \colon \rho \mapsto \rho', 
\quad \rho'_\alpha = g_{\sft(\alpha)} \circ \rho_\alpha \circ g_{\sfs(\alpha)}^{-1}.
\]
Let $\ov{\sfQ^a} = \{\, \ov{\alpha} \mid \alpha \in \sfQ^a \,\}$ be a copy of $\sfQ^a$
and define $\sfs, \sft \colon \ov{\sfQ^a} \to \sfQ^v$ by
$\sfs(\ov{\alpha})=\sft(\alpha)$, $\sft(\ov{\alpha})=\sfs(\alpha)$.
Set $\ov{\sfQ}=(\sfQ^v,\ov{\sfQ^a},\sfs,\sft)$,
which is the quiver obtained from $\sfQ$ by reversing the orientation of each arrow.
Then the representation space
\[
\Rep_{\ov{\sfQ}}(V) = \bigoplus_{\ov{\alpha} \in \ov{\sfQ^a}} 
\Hom(V_{\sfs(\ov{\alpha})},V_{\sft(\ov{\alpha})})
= \bigoplus_{\alpha \in \sfQ^a} \Hom(V_{\sft(\alpha)},V_{\sfs(\alpha)})
\]
is dual to the vector space $\Rep_\sfQ(V)$ via the trace pairing,
and the representation space of the {\em double}
\[
\wh{\sfQ} = (\sfQ^v,\wh{\sfQ}^a,\sfs,\sft), \quad 
\wh{\sfQ}^a :=  \sfQ^a \sqcup \ov{\sfQ^a}
\]
of $\sfQ$ may be identified with the cotangent bundle of $\Rep_\sfQ(V)$:
\[
\Rep_{\wh{\sfQ}}(V) = \Rep_\sfQ(V) \oplus \Rep_{\ov{\sfQ}}(V)
\simeq \Rep_\sfQ(V) \oplus \Rep_\sfQ(V)^* \simeq T^* \Rep_\sfQ(V).
\]
The group $G_V$ naturally acts on it and preserves 
the canonical symplectic form 
\[
\omega = \sum_{\alpha \in \sfQ^a} \tr d\Xi_\alpha \wedge d\Xi_{\ov{\alpha}},
\]
where 
$\Xi_\alpha \colon \Rep_{\wh{\sfQ}}(V) \to \Hom(V_{\sfs(\alpha)},V_{\sft(\alpha)})$,
$\Xi \mapsto \Xi_\alpha$ denotes the projection.
Extend the map $\sfQ^a \to \ov{\sfQ^a}$, $\alpha \mapsto \ov{\alpha}$ 
to an involution of $\wh{\sfQ}^a$ and 
define $\epsilon \colon \wh{\sfQ}^a \to \{ \pm 1 \}$ by 
$\epsilon (\sfQ^a)=1$ and $\epsilon(\ov{\sfQ^a})=-1$.
Then $\omega$ is also expressed as 
\[
\omega = \frac12 \sum_{\alpha \in \wh{\sfQ}^a} 
\epsilon(\alpha) \tr d\Xi_{\alpha} \wedge d\Xi_{\ov{\alpha}}.
\]
Let $\frg_V = \bigoplus_{i \in \sfQ^v} \gl(V_i)$ 
be the Lie algebra of $G_V$. 
A moment map on $\Rep_{\wh{\sfQ}}(V)$ for the $G_V$-action is given by
\[
\mu=(\mu_i)_{i \in \sfQ^v} \colon \Rep_{\wh{\sfQ}}(V) \to \frg_V,
\quad \mu_i(\Xi) := \sum_{\alpha \in \wh{\sfQ}^a; \sft(\alpha)=i}
\epsilon(\alpha)\Xi_\alpha \Xi_{\ov{\alpha}},
\]
where we identify $\frg_V^*$ with $\frg_V$ using the trace.

\begin{dfn}
For $\Xi \in \Rep_{\wh{\sfQ}}(V)$,
a collection $W=(W_i)_{i \in \sfQ^v}$ 
of subspaces $W_i \subset V_i$ 
is called a {\em $\Xi$-invariant} subspace
if $\Xi_\alpha(W_{\sfs(\alpha)}) \subset W_{\sft(\alpha)}$ 
for all $\alpha \in \wh{\sfQ}^a$. 

A point $\Xi \in \Rep_{\wh{\sfQ}}(V)$ is {\em stable} if 
the representation $(V,\Xi)$ of $\wh{\sfQ}$ is irreducible,
i.e., there exist no $\Xi$-invariant subspaces except the trivial ones $0, V$.
\end{dfn}

See \cite{King} for the fact that the above stability 
is one of Mumford's stability conditions~\cite{MFK}; 
therefore the stable points form a $G_V$-invariant Zariski open subset of $\Rep_{\wh{\sfQ}}(V)$ 
on which the quotient group $G_V/\C^\times$ acts freely and properly
(note that the subgroup $\C^\times \subset G_V$ acts trivially on $\Rep_{\wh{\sfQ}}(V)$).

\begin{dfn}
For $\zeta =(\zeta_i)_{i \in \sfQ^v} \in \C^{\sfQ^v}$, 
set $\zeta_V = (\zeta_i 1_{V_i})_{i \in \sfQ^v} \in \frg_V^{G_V}$ and define
\[
\frM^s_\sfQ(V,\zeta) = \{\, \Xi \in \mu^{-1}(\zeta_V) \mid \text{$\Xi$ is stable}\,\}/G_V,
\]
which we call the {\em quiver variety}.
\end{dfn}

By the definition $\frM^s_\sfQ(V,\zeta)$ is a complex symplectic manifold 
(if non-empty).
We also denote it by $\frM^s_\sfQ(\bv,\zeta)$, where 
$\bv = (v_i)_{i \in \sfQ^v} \in \Z_{\geq 0}^{\sfQ^v}$, $v_i := \dim V_i$  
is the {\em dimension vector} of $V$.

\begin{rem}
(i) If $\mu^{-1}(\zeta_V)$ is non-empty, then $\zeta$ must satisfy 
\[
\zeta \cdot \bv := \sum_{i \in \sfQ^v} \zeta_i v_i =0.
\]
Indeed, for $\Xi \in \mu^{-1}(\zeta_V)$, we have
\begin{align*}
\zeta \cdot \bv  
= \sum_{i \in \sfQ^v} \tr \mu_i(\Xi) 
= \sum_{\alpha \in \wh{\sfQ}^a} \epsilon(\alpha) \tr \Xi_\alpha \Xi_{\ov{\alpha}}
&= \sum_{\alpha \in \wh{\sfQ}^a} \epsilon(\ov{\alpha}) \tr \Xi_{\ov{\alpha}} \Xi_\alpha \\
&= -\sum_{\alpha \in \wh{\sfQ}^a} \epsilon(\alpha) \tr \Xi_\alpha \Xi_{\ov{\alpha}},
\end{align*}
and hence $\zeta \cdot \bv =0$.

(ii) If $\frM^s_\sfQ(\bv,\zeta)$ is non-empty, 
then the dimension formula for symplectic quotients shows
\[
\dim \frM_\sfQ^s(\bv,\zeta) = 2 \Delta(\bv), \quad 
\Delta(\bv) := \sum_{\alpha \in \sfQ^a} v_{\sfs(\alpha)}v_{\sft(\alpha)} 
- \sum_{i \in \sfQ^v} v_i^2 + 1.
\]
\end{rem}

\subsection{Triangular decomposition of $B_k$-coadjoint orbits}\label{subsec:tri}

Fix a non-zero irregular type $T \in \frt(\hcalK_t)/\frt(\hcalO_t)$
at some $t \in \bbP^1$
and let $k>1$ be the pole order of $dT$.
In this subsection we give a ``triangular decomposition'' of 
the $B_k$-coadjoint orbit $O_B$ through $dT$,
which enables us to describe $O_B$ as the representation space 
of the double of some quiver.

Take a local coordinate $z$ vanishing at $t$ and 
write 
\[
T(z)=\sum_{i=1}^{k-1}T_i z^{-i}, \quad 
dT=\sum_{i=1}^{k-1}(dT)_i z^{-i-1}dz,
\quad (dT)_i=-iT_i \in \frt.
\]
For $i=0,1, \dots , k-2$,
let $J_i$ be the set of simultaneous eigenspaces 
of ($T_{i+1}$, $T_{i+2}, \dots , T_{k-1}$) 
and $\C^n = \bigoplus_{p \in J_i} V^{(i)}_p$
the associated decomposition.
For each $i$ and $j \leq i$, 
we have a natural surjection $\pi_i \colon J_j \to J_i$
such that $V^{(j)}_p \subset V^{(i)}_{\pi_i(p)}$ for $p \in J_j$.

Fix a total ordering on $J_i$ for each $i$ such that
\[
i<k-2,\ p, q \in J_i, \ \pi_{i+1}(p) < \pi_{i+1}(q) \Longrightarrow p < q.
\]
Each decomposition $\C^n = \bigoplus_{p \in J_i} V^{(i)}_p$
induces a block decomposition of $\Mat_n(\C)$.
Let $\frp_i^+ (\text{resp.}\ \frp_i^-) \subset \Mat_n(\C)$ be 
the subset of block upper (resp.\ lower) triangular matrices;
explicitly, 
\[
\frp_i^+ = \bigoplus_{\sumfrac{p,q \in J_i}{p \geq q}} 
\Hom (V^{(i)}_p,V^{(i)}_q), 
\quad
\frp_i^- = \bigoplus_{\sumfrac{p,q \in J_i}{p \leq q}} 
\Hom (V^{(i)}_p,V^{(i)}_q).
\]
We have a decomposition $\frp_i^\pm = \frh_i \oplus \fru_i^\pm$,
where
\[
\frh_i = \bigoplus_{p \in J_i} 
\End (V^{(i)}_p) = \bigcap_{j=i+1}^{k-1} \Ker \ad_{T_j}, 
\quad
\fru_i^\pm = \bigoplus_{\sumfrac{p,q \in J_i}{p \gtrless q}} 
\Hom (V^{(i)}_p,V^{(i)}_q).
\]
Note that the property of our total orderings implies 
\begin{gather*}
\frp_{i+1}^\pm = \frp_i^\pm \oplus (\frh_{i+1} \cap \fru_i^\mp), 
\quad
\fru_i^\pm = \fru_{i+1}^\pm \oplus (\frh_{i+1} \cap \fru_i^\pm), \\
\fru_i^\pm \cap \Ker \ad_{T_{i+1}} \subset \fru_{i+1}^\pm.
\end{gather*}
For convenience we use the convention 
$\frp_{k-1}^\pm = \frh_{k-1} = \frg$, 
$\fru_{k-1}^\pm =0$.
The following lemma immediately follows from the arguments 
in the proof of \lmmref{lem:exponent}.

\begin{lmm}
An element $b(z)=\sum_{i=0}^{k-1} b_i z^i$ of $B_k$
stabilizes $dT$ if and only if $b_i \in \frh_i$ for all $i=1,2,\dots ,k-1$.
\end{lmm}

Set
\[
\calP_\pm = \set{b(z) = \sum_{i=0}^{k-1} b_i z^i \in B_k}%
{b_i \in \frp_i^\pm \ (i=1,2, \dots ,k-1)}.
\]
These are Lie subgroups of $B_k$ since $\frp_i \subset \frp_{i+1}$.
Let $\frP_\pm$ be their Lie algebras.
The pairing on $\frb_k \otimes \frb_k^*$ enables us 
to identify the duals of $\frP_\pm$ with
\[
\frP_\pm^* := \set{X(z) = \sum_{i=1}^{k-1} X_i z^{-i-1}dz \in \frb_k^*}%
{X_i \in \frp_i^\mp \ (i=1,2, \dots ,k-1)}.
\]
Similarly, set
\[
\calU_\pm = \set{b(z) = \sum_{i=0}^{k-1} b_i z^i \in B_k}%
{b_i \in \fru_i^\pm \ (i=1,2, \dots ,k-1)}.
\]
Note that these are not Lie subgroups of $B_k$.
Let $\frU_\pm = T_1 \calU_\pm$ be the tangent spaces at the identity,
whose duals are identified with
\[
\frU_\pm^* := \set{X(z) = \sum_{i=1}^{k-1} X_i z^{-i-1}dz \in \frb_k^*}%
{X_i \in \fru_i^\mp \ (i=1,2, \dots ,k-1)}.
\]

\begin{lmm}\label{lem:dec}
For any $b \in B_k$, there uniquely exist $b_- \in \calU_-$ and $b_+ \in \calP_+$
such that $b=b_- b_+$.
\end{lmm}

\begin{proof}
Write 
\[
b=\sum_{i=0}^{k-1}b_i z^i,
\quad
b_\pm=\sum_{i=0}^{k-1}b^\pm_i z^i,
\quad b_0=b^\pm_0=1_{\C^n}.
\]
In terms of the coefficients $b_i, b^\pm_i$, 
the relation $b=b_- b_+$ is expressed as
\[
b^-_i + b^+_i = b_i - \sum_{j=1}^{i-1} b^-_j b^+_{i-j}
\quad (i=1,2,\dots ,k-1).
\]
The above inductively defines $b^-_i \in \fru_i^-, b^+_i \in \frp_i^+$
since $\frg = \fru_i^- \oplus \frp_i^+$.
\end{proof}

For $B \in O_B$, take $b \in B_k$ such that $B=\Ad^*_b (dT)$
and decompose $b = b_- b_+$ as above.
Note that $b_-$ does not depend on the choice of $b$ as 
the stabilizer of $dT$ is contained in $\calP_+$.
Set $B' = (b_-^{-1} B)|_{\frb_k^*}$ 
(which is the $z^{-1}\frg[z^{-1}]dz$-part of the product $b_-^{-1}B$) and 
\[
Q = b_- - 1_{\C^n} \in \frU_-, \quad P = B'|_{\frU_-^*} \in \frU_-^*.
\]

\begin{thm}\label{thm:tri}
The map 
\[
O_B \to \frU_- \times \frU_-^*; \quad
B \mapsto (Q, P)
\]
is a symplectomorphism from $O_B$ onto the space 
$\frU_- \times \frU_-^*$ equipped with the symplectic form
$\res_{z=0} \tr dQ \wedge dP$.
\end{thm}

To prove it we need two lemmas.

\begin{lmm}\label{lem:tri}
For any $(Q,P) \in \frU_- \times \frU_-^*$, 
there exists a unique $B' \in \frb_k^*$ such that 
$B'|_{\frU_-^*} = P$ and $B'(1_{\C^n}+Q) - dT \in \frU_-^*$.
\end{lmm}

\begin{proof}
Write 
\[
Q=\sum_{i=1}^{k-1} Q_i z^i,
\quad
P=\sum_{i=1}^{k-1} P_i z^{-i-1}dz,
\quad
B'=\sum_{i=1}^{k-1} B'_i z^{-i-1}dz.
\]
Conditions $B'|_{\frU_-^*} = P$ and $B'(1_{\C^n}+Q) - dT \in \frU_-^*$ are then expressed as
\[
B'_{k-i} |_{\fru_{k-i}^+}=P_{k-i}, \quad
B'_{k-i}|_{\frp_{k-i}^-} = (dT)_{k-i} - \sum_{j=1}^{i-1} B'_{k-j} Q_{i-j} |_{\frp_{k-i}^-}
\quad
(i=1,2,\dots ,k-1).
\]
Since $\frg = \fru_{k-i}^+ \oplus \frp_{k-i}^-$,
the above inductively defines $B'_{k-i}$.
\end{proof}

\begin{lmm}\label{lem:P}
The $\calP_+$-orbit through $dT$ coincides with the subset $dT + \frU_-^* \subset \frb_k^*$.
\end{lmm}

\begin{proof}
For $B=\sum_{i=1}^{k-1}B_i z^{-i-1}dz \in \frb_k^*$ 
and $b=\sum_{i=0}^{k-1}b_i z^i \in \calP_+$,
the relation $B=\Ad^*_b(dT)$ 
is equivalent to
\[
B_{k-i} - (dT)_{k-i} + \sum_{j=1}^{i-1} (B_{k-j}b_{i-j} - b_{i-j}(dT)_{k-j}) = 0
\quad (i=1,2,\dots ,k-1).
\]
In terms of $X_i := B_i - (dT)_i\ (i=1,2,\dots ,k-1)$,
the above is expressed as
\begin{equation}
X_{k-i} = \sum_{j=1}^{i-1} ([b_{i-j},(dT)_{k-j}]-X_{k-j}b_{i-j})
\quad (i=1,2,\dots ,k-1).
\label{eq:P}
\end{equation}
Now assume \eqref{eq:P}. Then $X_{k-1}=0 \in \fru_{k-1}^+$, and furthermore we have
\begin{equation}
\begin{aligned}
&[\frp_{i-j}^+,(dT)_{k-j}] \subset 
[\frp_{k-1-j}^+,(dT)_{k-j}] \subset
\fru_{k-j-1}^+ \subset \fru_{k-i}^+, \\
&\fru_{k-j}^+ \cdot \frp_{i-j}^+ \subset \fru_{k-j}^+ \cdot \frp_{k-j}^+ 
\subset \fru_{k-j}^+ \subset \fru_{k-i}^+
\label{eq:P2}
\end{aligned}
\end{equation}
for $j<i<k$. Therefore \eqref{eq:P} inductively shows 
$X_{k-i} \in \fru_{k-i}^+$ for all $i$.
Hence $B-dT \in \frU_-^*$.

Conversely, assume $B-dT \in \frU_-^*$. 
To find $b \in \calP_+$ satisfying \eqref{eq:P}, 
we first project \eqref{eq:P} to 
$\fru_{k-i}^+ \ominus \fru_{k-i+1}^+ = \frh_{k-i+1} \cap \fru_{k-i}^+$
for $i \geq 2$.
By \eqref{eq:P2}, we then obtain
\begin{align*}
X_{k-i} |_{\fru_{k-i}^+ \ominus \fru_{k-i+1}^+} &= 
\sum_{j=1}^{i-1} ([b_{i-j},(dT)_{k-j}]-X_{k-j}b_{i-j})|_{\fru_{k-i}^+ \ominus \fru_{k-i+1}^+} \\
&= [b_1,(dT)_{k-i+1}]|_{\fru_{k-i}^+ \ominus \fru_{k-i+1}^+}
\quad (i=2,\dots ,k-1).
\end{align*}
Note that $[b_1,(dT)_{k-i+1}]|_{\fru_{k-i}^+ \ominus \fru_{k-i+1}^+}$ depends only on
the $\fru_{k-i}^+ \ominus \fru_{k-i+1}^+$-component of $b_1$ 
since 
\[
\frp_1^+ \ominus \fru_{k-i}^+ \subset \Ker \ad_{T_{k-i+1}}, \quad  
[\fru_{k-i+1}^+, (dT)_{k-i+1}] \subset \fru_{k-i+1}^+.
\]
By $\fru_{k-i}^+ \cap \Ker \ad_{T_{k-i+1}} \subset \fru_{k-i+1}^+$, 
we have $\fru_{k-i}^+ \ominus \fru_{k-i+1}^+ \subset \range \ad_{T_{k-i+1}}$.
Hence there exists a unique $b_1 \in \fru_1$ such that 
$b_1 |_{\fru_{k-i}^+ \ominus \fru_{k-i+1}^+}$
satisfies the above equality for $i=2,\dots ,k-1$.

Next, we project \eqref{eq:P} to 
$\fru_{k-i+1}^+ \ominus \fru_{k-i+2}^+$
for $i \geq 3$.
By \eqref{eq:P2}, we then obtain
\begin{align*}
X_{k-i} |_{\fru_{k-i+1}^+ \ominus \fru_{k-i+2}^+} &= 
\sum_{j=1}^{i-1} ([b_{i-j},(dT)_{k-j}]-X_{k-j}b_{i-j})|_{\fru_{k-i+1}^+ \ominus \fru_{k-i+2}^+} \\
&= ([b_1,(dT)_{k-i+1}]-X_{k-i+1}b_1
+ [b_2,(dT)_{k-i+2}])|_{\fru_{k-i+1}^+ \ominus \fru_{k-i+2}^+}
\end{align*}
for $i=3,\dots ,k-1$.
Note that $[b_2,(dT)_{k-i+2}]|_{\fru_{k-i+1}^+ \ominus \fru_{k-i+2}^+}$ depends only on
the $\fru_{k-i+1}^+ \ominus \fru_{k-i+2}^+$-component of $b_2$ 
since 
\[
\frp_2^+ \ominus \fru_{k-i+1}^+ \subset \Ker \ad_{T_{k-i+2}}, \quad 
[\fru_{k-i+2}^+, (dT)_{k-i+2}] \subset \fru_{k-i+2}^+.
\]
We have $\fru_{k-i+1}^+ \ominus \fru_{k-i+2}^+ \subset \range \ad_{T_{k-i+2}}$.
Hence there exists a unique $b_2 \in \fru_2^+$ such that 
$b_2 |_{\fru_{k-i+1}^+ \ominus \fru_{k-i+2}^+}$
satisfies the above equality for $i=3,\dots ,k-1$.

Repeating this argument,
we see that \eqref{eq:P} uniquely determines $b_i \in \fru_i^*$, $i=1,2, \dots ,k-2$,
which shows the existence (and uniqueness) of $b \in \calU_+ \subset \calP_+$.
\end{proof}

\begin{proof}[Proof of \thmref{thm:tri}]
Lemmas~\ref{lem:tri} and \ref{lem:P} imply that the map $B \mapsto (Q,P)$ 
is bijective (note that $B$ is uniquely determined from $Q$ and $B'$). 
We show that it is also symplectic.
Let $B=\Ad^*_b(dT) \in O_B$.
Decompose $b=b_- b_+$ as in \lmmref{lem:dec} and let $B_+=\Ad^*_{b_+}(dT)$.
For $i=1,2$, let $v_i \in T_B O_B$ and 
$(\xi_i,\eta_i)$ the corresponding tangent vector at $(Q,P) \in \frU_- \times \frU_-^*$.
Take $X_i \in \frb_k$ such that $[X_i,B]=v_i$.
By differentiating $b=b_-b_+$, we find $\xi'_i \in \frP_+$ such that 
\[
X_i = (\xi_i + b_- \xi'_i)b_-^{-1}.
\]
Now we calculate the symplectic form evaluated at $v_1,v_2$:
\begin{equation}
\begin{aligned}
\omega_B(v_1,v_2) &= \res_{z=0} \tr (B[X_1,X_2]) \\
&= \res_{z=0} \tr (B[(\xi_1 + b_- \xi'_1)b_-^{-1},(\xi_2 + b_- \xi'_2)b_-^{-1}]) \\
&= \res_{z=0} \tr (B[\xi_1 b_-^{-1},\xi_2 b_-^{-1}]) + 
\res_{z=0} \tr (B[b_- \xi'_1 b_-^{-1},b_- \xi'_2 b_-^{-1}]) \\
&\quad + \res_{z=0} \tr (B[\xi_1 b_-^{-1}, b_- \xi'_2 b_-^{-1}]) + 
\res_{z=0} \tr (B[b_- \xi'_1 b_-^{-1},\xi_2 b_-^{-1}]).
\label{eq:form}
\end{aligned}
\end{equation}
\lmmref{lem:P} implies $B_+ -dT \in \frU_-^*$; 
thus the second term in the most right hand side of \eqref{eq:form} is simplified as follows:
\[
\res_{z=0} \tr (B[b_- \xi'_1 b_-^{-1},b_- \xi'_2 b_-^{-1}])
= \res_{z=0} \tr (B_+[\xi'_1,\xi'_2])
= \res_{z=0} \tr (dT[\xi'_1,\xi'_2]).
\]
Since the coefficient of $[dT,\xi_1']$ in $z^{-i-1}dz$ ($i \geq 1$)
is contained in
\begin{align*}
\sum_{j=1}^{k-i-1} [(dT)_{i+j},\frp_j^+] 
\subset \sum_{j=1}^{k-i-1} [(dT)_{i+j},\frp_{i+j-1}^+] 
&= \sum_{j=1}^{k-i-1} [(dT)_{i+j},\fru_{i+j-1}^+] \\
&\subset \sum_{j=1}^{k-i-1} \fru_{i+j-1}^+ \subset \fru_i^+,
\end{align*}
we find $[dT,\xi_1'] \in \frU_-^*$ and thus 
\[
\res_{z=0} \tr (dT[\xi'_1,\xi'_2])=\res_{z=0} \tr ([dT,\xi'_1]\xi'_2])=0.
\]
Next, look at the third term in the most right hand side of \eqref{eq:form}:
\begin{align*}
\res_{z=0} \tr (B[\xi_1 b_-^{-1}, b_- \xi'_2 b_-^{-1}])
&=\res_{z=0} \tr (B_+[b_-^{-1}\xi_1, \xi'_2]) \\
&=\res_{z=0} \tr ([\xi'_2,B_+]b_-^{-1}\xi_1).
\end{align*}
Differentiating the relation $P=b_-^{-1}B|_{\frU_-^*}$ yields
\[
[\xi'_i,B_+]b_-^{-1}|_{\frU_-^*} = \eta_i + b_-^{-1} B \xi_i b_-^{-1}|_{\frU_-^*}. 
\]
Hence the third term is expressed as
\begin{equation}
\res_{z=0} \tr ([\xi'_2,B_+]b_-^{-1}\xi_1)
= \res_{z=0} \tr (\eta_2 \xi_1) + \res_{z=0} \tr (B \xi_2 b_-^{-1} \xi_1 b_-^{-1}).
\end{equation}
Similarly, the fourth term is
\begin{equation}
\res_{z=0} \tr (B[b_- \xi'_1 b_-^{-1},\xi_2 b_-^{-1}])=
-\res_{z=0} \tr (\eta_1 \xi_2) - \res_{z=0} \tr (B \xi_1 b_-^{-1} \xi_2 b_-^{-1}).
\end{equation}
Thus we obtain
\begin{align*}
\omega_B(v_1,v_2) &= \res_{z=0} \tr (B[\xi_1 b_-^{-1},\xi_2 b_-^{-1}]) + 
\res_{z=0} \tr (\eta_2 \xi_1) + \res_{z=0} \tr (B \xi_2 b_-^{-1} \xi_1 b_-^{-1}) \\
&\quad -\res_{z=0} \tr (\eta_1 \xi_2) - \res_{z=0} \tr (B \xi_1 b_-^{-1} \xi_2 b_-^{-1}) \\
&=\res_{z=0} \tr (\xi_1 \eta_2 - \xi_2 \eta_1),
\end{align*}
which shows $\omega = \res_{z=0} \tr dQ \wedge dP$.
\end{proof}

\subsection{Quiver description of $B_k$-coadjoint orbits}\label{subsec:quiver-B}

Using \thmref{thm:tri}, we will describe $O_B$ as the representation space 
of the double of the following quiver $\sfQ$.
Let $\sfQ^v$ be the set of simultaneous eigenspaces for 
coefficients $(T_1,T_2,\dots ,T_{k-1})$ of $T$
and $V =(V_p)_{p \in \sfQ^v}$ the collection of the corresponding simultaneous eigenspaces:
\[
\sfQ^v = J_0, \qquad V_p=V^{(0)}_p \quad (p \in \sfQ^v).
\] 
We have a natural surjection 
$\pi_i \colon \sfQ^v \to J_i$
for $i=0,1,\dots ,k-2$.
For each $p,q \in \sfQ^v$ with $p<q$,
draw $m_{p,q}$ arrows from $p$ to $q$, where 
\[
m_{p,q} := \max\{\, i \mid 
\pi_i(p)<\pi_i(q)\,\} \in \{ 0,1, \dots ,k-2 \}.
\]
Note that if we write 
\[
T(z)= \bigoplus_{p \in \sfQ^v} \left( t_p(z)\,1_{V_p} \right),
\quad t_p(z) \in z^{-1}\C[z^{-1}],
\]
then $m_{p,q} = \deg_{1/z}(t_p(z)-t_q(z))-1$ (each $t_p$ is called an {\em eigenvalue} of $T$).

Let $\sfQ$ be the resulting quiver 
and label the arrows from $p$ to $q$ as $\alpha_{q p;i}$, $i=1,2,\dots ,m_{p,q}$.
Then $G_V = H$ and  
\begin{equation}
\begin{aligned}
\frU_- = \bigoplus_{i=1}^{k-1} \fru^-_i z^i 
&= \bigoplus_{i=1}^{k-1} \bigoplus_{\sumfrac{p,q \in J_i}{p<q}} 
\Hom (V^{(i)}_p,V^{(i)}_q) z^i 
\\
&= \bigoplus_{i=1}^{k-1} 
\bigoplus_{\sumfrac{p,q \in J_0}{\pi_i(p)<\pi_i(q)}} 
\Hom (V_p,V_q) z^i
\simeq \Rep_{\sfQ}(V),
\end{aligned}\label{eq:quiver-U}
\end{equation}
where the last isomorphism sends the factor $\Hom (V_p,V_q) z^i$
to the set of linear maps associated to the arrow $\alpha_{q p;i}$.
Therefore \thmref{thm:tri} implies the following:
\begin{crl}\label{cor:quiver-B}
The map given in $\eqref{eq:quiver-U}$ induces 
a $G_V$-equivariant symplectic isomorphism between 
the $B_k$-coadjoint orbit $O_B$ and $\Rep_{\wh{\sfQ}}(V)$.
\end{crl}

\begin{example}
Here we show some examples of the quiver $\sfQ$ defined above.
For simplicity assume that $\frt$ is the standard maximal torus.

i) When $k\le 3$, many examples of $\sfQ$ are given in \cite{Boa08}.
In this case the degree in $1/z$ of the difference $t_p(z)-t_q(z)$ of 
each two distinct eigenvalues of $T$ is at most $2$.
Hence the underlying graph of $\sfQ$ is always simply-laced.

ii) When $k=4$, vertices of $\sfQ$ may be joined by double arrows.
In the simplest case where $n=2$ and 
$T(z)=\sum_{i=1}^3 T_i z^{-i}$, 
$T_3=\begin{pmatrix}
	a&0\\
	0&b
\end{pmatrix}$, $a \neq b$,
the quiver $\sfQ$ is given as follows.
\[
\begin{xy}
*\cir<5pt>{}="A",
(20,0)*\cir<5pt>{}="B",
\ar@{->}@<0.7mm> "A";"B",
\ar@{->}@<-0.7mm> "A";"B",
\end{xy}
\]

iii) Finally we show an example for the case $k=5$,
where vertices of $\sfQ$ may be joined by triple arrows.
Consider an irregular type $T=\sum_{i=1}^{4}T_{i}z^{-i}$ with $n=4$ of the form
\begin{align*}
	T_{4}&=\mathrm{diag\,}(a^{(4)}_{1},a^{(4)}_{2},a^{(4)}_{2},
	a_{2}^{(4)}),&
	T_{3}&=\mathrm{diag\,}(\ast,a^{(3)}_{1},a^{(3)}_{2},
	a^{(3)}_{2}),\\
	T_{2}&=\mathrm{diag\,}(\ast,\ast,a^{(2)}_{1},a_{2}^{(2)}),&
	T_{1}&=\mathrm{diag\,}(\ast,\ast,\ast,\ast),
\end{align*}
where $a^{(i)}_{1}\neq a^{(i)}_{2}$.
The associated quiver $\sfQ$ is then given as follows.
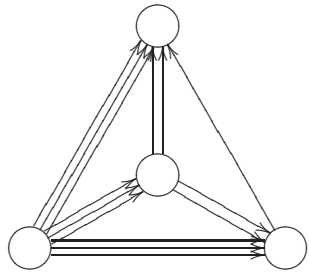
\[
\begin{xy}
(-17,0)*\cir<9pt>{}="A",
(17,0)*\cir<9pt>{}="B",
(0,29.5)*\cir<9pt>{}="C",
(0,9.7)*\cir<9pt>{}="D",
\ar@{->}@<1mm> "A";"B",
\ar@{->}@<-1mm> "A";"B",
\ar@{->} "A";"B",
\ar@{->}@<1mm> "A";"C",
\ar@{->}@<-1mm> "A";"C",
\ar@{->} "A";"C",
\ar@{->}@<1mm> "A";"D",
\ar@{->}@<-1mm> "A";"D",
\ar@{->} "A";"D",
\ar@{->}@<0.7mm> "D";"B",
\ar@{->}@<-0.7mm> "D";"B",
\ar@{->}@<0.7mm> "D";"C",
\ar@{->}@<-0.7mm> "D";"C",
\ar@{->}"B";"C"
\end{xy}
\]
\end{example}

The rest of this subsection is devoted to prove the following lemma,
which will be used in the next subsection.

\begin{lmm}\label{lem:stability}
Let $B=\sum_{i=1}^{k-1}B_i z^{-i-1}dz \in O_B$ and $\Xi \in \Rep_{\wh{\sfQ}}(V)$ 
the corresponding element under the above isomorphism.

{\rm (i)} For any $\Xi$-invariant subspace $W=(W_p)_{p \in J_0}$ of $V$,
the direct sum $S := \bigoplus_p W_p \subset \C^n$ is invariant under all $B_i$.

{\rm (ii)} Any subspace $S \subset \C^n$ invariant under all $B_i$
is homogeneous with respect to the decomposition $\C^n = \bigoplus_{p \in J_0} V_p$
and the collection $W=(W_p)_{p \in J_0}$, $W_p := S \cap V_p$ 
is a $\Xi$-invariant subspace of $V$.
\end{lmm}

The proof of it needs two sublemmas.

\begin{lmm}
Let $B = \sum_{j=1}^{k-1}B_j z^{-j-1}dz \in O_B$
and assume that for some $i=1,2 \dots , k$, 
\[
B_{k-j} \in \frh_{k-j} \quad (j \leq i).
\]
Then $B_{k-j}=(dT)_{k-j}$ for $j \leq i$.
\end{lmm}

\begin{proof}
Take $b=\sum_{j=0}^{k-1} b_j z^j \in B_k$ so that $B=\Ad^*_b(dT)$.
Then
\[
B_{k-j} - (dT)_{k-j} = \sum_{l=1}^{j-1} b_l (dT)_{k-j+l} - \sum_{l=1}^{j-1} B_{k-j+l} b_l.
\]
First, the equality for $j=2$ reads $B_{k-2} - (dT)_{k-2} = [b_1, (dT)_{k-1}]$.
Since the left hand side is contained in $\frh_{k-2}$ while the right hand side is in 
$\range \ad_{(dT)_{k-1}} = \frg \ominus \frh_{k-2}$, 
we see that both sides must be zero. Hence $B_{k-2}=(dT)_{k-2}$ and $b_1 \in \frh_{k-2}$.
If $i=2$ we are done; so assume $i>2$ and look at the equality for $j=3$:
\[
B_{k-3} - (dT)_{k-3} = [b_1,(dT)_{k-2}]+[b_2,(dT)_{k-1}].
\]
Since $B_{k-3} - (dT)_{k-3} \in \frh_{k-3}$, $[b_1,(dT)_{k-2}] \in \frh_{k-2} \ominus \frh_{k-3}$
and $[b_2,(dT)_{k-1}] \in \frg \ominus \frh_{k-2}$, 
we see that all the terms must be zero 
and hence $B_{k-3} = (dT)_{k-3}$, $b_1 \in \frh_{k-3}$, $b_2 \in \frh_{k-2}$.
Repeating this argument, we finally obtain $B_{k-i}=(dT)_{k-i}$ 
(and $b_l \in \frh_{k-i+l-1}$ for $l<i$).
\end{proof}

\begin{lmm}
Let $B=\sum_{j=1}^{k-1}B_j z^{-j-1}dz \in O_B$.
Then there exists a unique tuple $(b^{(1)}, b^{(2)}, \dots , b^{(k-2)})$ 
of elements of $B_k$ satisfying the following conditions:
\begin{enumerate}
\item Each $b^{(i)}$ has the form
\[
b^{(i)}(z) = \exp (z^{k-1}X^{(i)}_{k-1}) \exp (z^{k-2}X^{(i)}_{k-2}) 
\cdots \exp (z X^{(i)}_1),
\quad X^{(i)}_j \in \frh_{k-i} \ominus \frh_{k-i-1}.
\]
\item Define $B^{(i)} = \sum_{j=1}^{k-1}B^{(i)}_j z^{-j-1}dz \in O_B$, $i=0,1, \dots , k-2$ 
inductively by
\[
B^{(0)} = B, \quad B^{(i)} = \Ad^*_{b^{(i)}}(B^{(i-1)}).
\]
Then $B^{(i)}_j \in \frh_{k-i-1}$ for all $i,j$ and
$B^{(i)}_{k-j} = (dT)_{k-j}$ for $j \leq i+1$.
\end{enumerate}
\end{lmm}

\begin{proof}
By \cite[Proposition~9.3.2]{BV} (see Proposition~\ref{prop:BV1}), we see that 
there exists a unique $b^{(1)} \in B_k$ of the form
\[
b^{(1)}(z) = \exp (z^{k-1}X^{(1)}_{k-1}) \exp (z^{k-2}X^{(1)}_{k-2}) 
\cdots \exp (z X^{(1)}_1),
\quad X^{(1)}_j \in \frg \ominus \frh_{k-2}
\]
such that $B^{(1)}_{k-1} = B_{k-1} = (dT)_{k-1}$ and $B^{(1)}_j \in \frh_{k-2}$ for all $j$.
Since $B^{(1)} \in O_B$, 
the previous lemma shows $B^{(1)}_{k-2} = (dT)_{k-2}$.

Now assume $i>1$ and the unique existence of $b^{(l)},\ l=1,2, \dots , i-1$ of the required form 
satisfying that $B^{(l)}_j \in \frh_{k-i-1}$ for all $l < i$ and $j$, and that
$B^{(l)}_{k-j} = (dT)_{k-j}$ for $j \leq l+1 \leq i$.
Then set
\[
B' = B^{(i-1)} - \sum_{j=1}^{i-1} B^{(i-1)}_{k-j} z^{j-k-1}dz 
= B^{(i-1)} - \sum_{j=1}^{i-1} (dT)_{k-j} z^{j-k-1}dz \in \frb_k^*,
\]
which is expressed as $B'=\sum_{j=1}^{k-i} B'_j z^{-j-1}dz$ with 
$B'_{k-i} = (dT)_{k-i}$ and $B'_j \in \frh_{k-i}$ for all $j$.
Therefore \cite[Proposition~9.3.2]{BV} with $G=H_{k-i}:=\exp (\frh_{k-i})$ 
shows that there exists a unique $b^{(i)} \in B_k$ of the required form such that
\[
\Ad^*_{b^{(i)}}(B') = \sum_{j=1}^{k-i} B''_j z^{-j-1}dz, \quad
B''_{k-i} = (dT)_{k-i}, \quad B''_j \in \frh_{k-i-1}.
\]
Since $B^{(i)} = \Ad^*_{b^{(i)}}(B') + \sum_{j=1}^{i-1} (dT)_{k-j} z^{j-k-1}dz$ 
and hence $B^{(i)}_j \in \frh_{k-i-1}$ for all $j$, the previous lemma shows 
$B^{(i)}_{k-i-1} = (dT)_{k-i-1}$.
The proof is completed.
\end{proof}

\begin{proof}[Proof of Lemma~\ref{lem:stability}]
We first show that any subspace $S \subset \C^n$ invariant under all $B_l$
is homogeneous with respect to the decomposition 
$\C^n = \bigoplus_{p \in J_0}V_p$.
Let $b^{(i)}, X^{(i)}_l, B^{(i)}=\sum_{l=1}^{k-1}B^{(i)}_l z^{-l-1}dz$ be as in the previous lemma.
We claim that $S$ is invariant under $B^{(i)}_l$ for all $i,l$,
which implies the desired homogeneity of $S$,
and prove it by the induction on $i$.
The case $i=0$ is clear; so assume $i>0$ and 
that $S$ is invariant under $B^{(l)}_j$ for all $l<i$ and $j$.
Define $B^{(i,j)}=\sum_{l=1}^{k-1} B^{(i,j)}_l z^{-l-1}dz \in O_B$, $j=0,1, \dots , k-1$ 
inductively by
\[
B^{(i,0)}=B^{(i-1)}, \quad B^{(i,j)}= \Ad^*_{\exp( z^j X^{(i)}_j)} (B^{(i,j-1)}).
\] 
Note that each $X^{(i)}_j$ commutes with $B^{(i-1)}_{k-l} = (dT)_{k-l}$, $l<i$
as it is contained in $\frh_{k-i}$. Hence
\[
B^{(i,j)}_{k-l} = B^{(i,j-1)}_{k-l}\quad (l<i+j), \quad
B^{(i,j)}_{k-j-i} = B^{(i,j-1)}_{k-j-i} + [X^{(i)}_j, B^{(i,j-1)}_{k-i}],
\]
which in particular implies 
\[
B^{(i)}_{k-l}=
\begin{cases}
B^{(i-1)}_{k-l} & (l \leq i), \\
B^{(i,l-i)}_{k-l} & (l>i).
\end{cases}
\]
We show that $S$ is invariant under $B^{(i,j)}_l$ for all $l$ by induction on $j$.
The case $j=0$ is clear.
Assume $j>0$ and that $S$ is invariant under $B^{(i,j')}_l$ for all $j'<j$ and $l$.
Then $S$ is invariant under $B^{(i,j)}_{k-l}$ for $l<i+j$.
Also, because 
$[X^{(i)}_j, B^{(i,j-1)}_{k-i}] = [X^{(i)}_j, (dT)_{k-i}] \in \frh_{k-i} \ominus \frh_{k-i-1}$,
the above expression for $B^{(i,j)}_{k-j-i} = B^{(i)}_{k-j-i} \in \frh_{k-i-1}$ implies 
\[
B^{(i,j)}_{k-j-i} = B^{(i,j-1)}_{k-j-i}|_{\frh_{k-i-1}}, 
\quad
[X^{(i)}_j, (dT)_{k-i}]  
= - B^{(i,j-1)}_{k-j-i}|_{\frh_{k-i} \ominus \frh_{k-i-1}}.
\]
Since $S$ is invariant under $(dT)_{k-l}=B^{(i-1)}_{k-l}$ for $l \leq i$ 
and $\ad_{(dT)_{k-i}}$ acts as a nonzero scalar on each direct summand
$\Hom(V^{(k-i-1)}_p, V^{(k-i-1)}_q)$ of $\frh_{k-i} \ominus \frh_{k-i-1}$,
we see from the above equalities that $S$ is invariant 
under both $B^{(i,j)}_{k-j-i}$ and $X^{(i)}_j$, and hence under $B^{(i,j)}_l$ for all $l$.

Now we have a one-to-one correspondence between 
the collections $W=(W_p)_{p \in J_0}$ of subspaces $W_p \subset V_p$
and the subspaces $S \subset \C^n$ homogeneous with respect to the 
decomposition $\C^n = \bigoplus_{p \in J_0} V_p$:
\[
W \mapsto S := \bigoplus_{p \in J_0} W_p,
\qquad
S \mapsto W =(W_p)_{p \in J_0},\ W_p := S \cap V_p. 
\]
We next show that under this correspondence,
$W$ is $\Xi$-invariant if and only if 
$S$ is invariant under all $B_i$.
For given $W$, take a subspace $W'_p \subset V_p$ complimentary to $W_p$ for each $p \in J_0$ 
and set $S' = \bigoplus_p W'_p$. 
Define a one-parameter subgroup $\tau \colon \C^\times \to H$ by
\[
\tau(u) = \begin{bmatrix} u 1_S & 0 \\ 0 & 1_{S'} \end{bmatrix} 
\colon \C^n = S \oplus S' \to S \oplus S'.
\] 
If we express each $\Xi_\alpha\ (\alpha \in \wh{\sfQ}^a)$ as 
\[
\Xi_\alpha = 
\begin{bmatrix} 
\Xi_\alpha^{11} & \Xi_\alpha^{12} \\ 
\Xi_\alpha^{21} & \Xi_\alpha^{22} 
\end{bmatrix} 
\colon V_{\sfs(\alpha)} = W_{\sfs(\alpha)} \oplus W'_{\sfs(\alpha)} \to 
W_{\sft(\alpha)} \oplus W'_{\sft(\alpha)} =V_{\sft(\alpha)},
\]
then the action by $\tau(u)$ transforms it to  
\[
\begin{bmatrix} 
\Xi_\alpha^{11} & u\Xi_\alpha^{12} \\ 
u^{-1}\Xi_\alpha^{21} & \Xi_\alpha^{22}
\end{bmatrix}.
\]
Hence the limit $\lim_{u \to 0} \tau(u) \cdot \Xi$
exists if and only if $\Xi_\alpha^{21}=0$ for all $\alpha$, i.e.,
$W$ is $\Xi$-invariant.
Similarly, if we express each $B_i$ as 
\[
B_i = 
\begin{bmatrix} 
B_i^{11} & B_i^{12} \\ 
B_i^{21} & B_i^{22} 
\end{bmatrix} 
\colon \C^n = S \oplus S' \to S \oplus S'= \C^n,
\]
then the limit $\lim_{u \to 0} \tau(u) B \tau(u)^{-1}$
exists in $O_B$ (note that $O_B \subset \frb_k^*$ is closed 
as $B_k$ is a unipotent algebraic group; see \cite[Theorem 2]{Ros})
if and only if $B_i^{21}=0$ for all $i$, i.e.,
$S$ is invariant under all $B_i$.

Since the map $O_B \simeq \Rep_{\wh{\sfQ}}(V)$, $B \mapsto \Xi$ 
is an $H$-equivariant homeomorphism,
the existence of $\lim_{u \to 0} \tau(u) \cdot \Xi$ 
is equivalent to that of $\lim_{u \to 0} \tau(u) B \tau(u)^{-1}$;
we are done.
\end{proof}

\subsection{Moduli spaces and quivers}\label{subsec:quiver-moduli}

As in \crlref{cor:moduli1}, 
let $\bfT =(T_t)_{t \in D}$ be a collection of irregular types and set 
\[
H_t = \{\, h \in G \mid h T_t h^{-1}=T_t\,\}, \quad \frh_t = \Lie H_t 
\quad (t \in D), 
\quad \bfH = \prod_{t \in D} H_t.
\]
Assume that $D_\sirr := \{\, t \in D \mid T_t \neq 0\,\}$
is non-empty as the empty case is not interesting.
For each $t \in D_\sirr$, 
let $\sfQ_t, V_t$ be the quiver and collection of vector spaces associated to $T_t$
(so $G_{V_t}=H_t$).
Take a base point $\infty \in D_\sirr$ and set $D_0 = D \setminus \{ \infty \}$.

\begin{crl}
The extended moduli space $\tcalM^*(\bfT)$
is $\bfH$-equivariantly symplectomorphic to the product
\[
\prod_{t \in D_0} T^* G 
\times \prod_{t \in D_\sirr} \Rep_{\wh{\sfQ}_t}(V_t).
\] 
Here $H_\infty$ acts on each copy of $T^*G$ as the standard action
coming from the right multiplication
and on $\Rep_{\wh{\sfQ}_\infty}(V_\infty)$ in the obvious way,
while for $t \in D_0$, 
the group $H_t$ acts on the $t$-th copy of $T^* G$ 
as the action coming from the left multiplication 
and on $\Rep_{\wh{\sfQ}_t}(V_t)$ (if $t \in D_\sirr$)
in the obvious way.
\end{crl}

\begin{proof}
Recall that if $M$ is a complex symplectic manifold 
on which $G$ acts in a Hamiltonian fashion with  
a moment map $\mu \colon M \to \frg$,
then the symplectic quotient of the product $T^*G \times M$
by the diagonal action of $G$ 
(where the $G$-action on $T^*G$ comes from the right multiplication)
is canonically symplectomorphic to $M$:
\[
(T^* G \times M) \GIT G \simeq M; \quad [a,R,x] \mapsto a \cdot x.
\] 
Under this map, the $G$-action on $M$ corresponds to 
the one on $(T^* G \times M) \GIT G$ 
defined by $f \cdot [a,R,x] = [fa,R,x]$.

Let $k$ be the pole order of $dT_\infty$ and $O_B$ the $B_k$-coadjoint orbit through $dT_\infty$.
For $t \in D$, let $\tO_t$ be the extended orbit associated to $T_t$.
Propositions \ref{prop:decouple} and \ref{prop:ext-moduli} together with the above fact imply
\[
\tcalM^*(\bfT) \simeq 
\left( \prod_{t \in D} \tO_t \right) 
\GIT G \simeq O_B \times \prod_{t \in D_0} \tO_t. 
\]
\crlref{cor:quiver-B} now shows the assertion.
\end{proof}

In the rest of this section we assume $D_\sirr = \{ \infty \}$. 
In this case, the space $\tcalM^*(\bfT)$ is 
$\bfH$-equivariantly symplectomorphic to the product
\[
\prod_{t \in D_0} T^* G \times \Rep_{\wh{\sfQ}_\infty}(V_\infty).
\] 
Hence the symplectic quotient $\calM^*(\bfT,\bfL)$ of $\tcalM^*(\bfT)$ 
(see \crlref{cor:moduli2})
by the $\bfH$-action along the coadjoint orbit through 
each $-\bfL = (-L_t)_{t \in D} \in \bigoplus_{t \in D} \frh_t^*$ 
is isomorphic to 
\[
\left( \prod_{t \in D_0} O(L_t) \times \Rep_{\wh{\sfQ}_\infty}(V_\infty) \right)
\GIT_{O(-L_\infty)} H_\infty,
\]
where $O(\pm L_t)$ is the $H_t$-coadjoint orbit through $\pm L_t$ for $t \in D$ 
(note that $H_t = G$ for $t \in D_0$).
By the shifting trick, we thus obtain
\[
\calM^*(\bfT,\bfL) \simeq 
\left( O(L_\infty) 
\times \prod_{t \in D_0} O(L_t) 
\times \Rep_{\wh{\sfQ}_\infty}(V_\infty) \right)
\GIT H_\infty.
\]
Denoting by $V_p$, $p \in \sfQ_\infty^v$ 
the simultaneous eigenspaces for the coefficients of $T_\infty(z)$,
we can express $O(L_\infty)$ as the product $\prod_p O_p(L_\infty)$,
where $O_p(L_\infty)$ is the $\GL(V_p)$-coadjoint orbit through $L_\infty |_{V_p}$.
Now recall that coadjoint orbits of general linear groups admit 
a sort of quiver description; see \secref{subsec:quiver-A}.
This fact leads to the definition of the following quiver $\sfQ$.

For each $t \in D_0$,
fix a marking  
$(\lambda_{t,1},\lambda_{t,2}, \dots , \lambda_{t,d_t})$
of $O(L_t)$, i.e., 
a tuple satisfying 
\[
\prod_{i=1}^{d_t} (\pi_\sres (L_t) - \lambda_{t,i} 1_{\C^n}) = 0
\]
(see \dfnref{dfn:marking}).
Also for each $p \in \sfQ_\infty^v$, fix a marking 
$(\lambda_{p,1},\lambda_{p,2}, \dots , \lambda_{p,d_p})$
of $O_p(L_\infty) \subset \gl(V_p)^*$.
Set
\[
\sfQ^v = \sfQ_\infty^v \sqcup \bigsqcup_{t \in D_0} \{\, [t,l] \mid l=1,2, \dots ,d_t-1\,\}
\sqcup \bigsqcup_{p \in \sfQ_\infty^v} \{\, [p,l] \mid l=1,2, \dots , d_p-1 \,\},
\]
and draw one arrow from
\begin{itemize}
\item each $[t,l]\, (l \geq 2)$ to $[t,l-1]$, 
\item each $[p,l]\, (l \geq 2)$ to $[p,l-1]$,
\item each $[t,1]$ to each $p \in \sfQ_\infty^v$, 
\item each $[p,1]$ to $p$.
\end{itemize}
Denote the set of these arrows by $\sfQ_0^a$, and define
\[
\sfQ^a = \sfQ_\infty^a \sqcup \sfQ_0^a.
\]

Define a collection $\zeta=(\zeta_i)_{i \in \sfQ^v}$ of complex numbers by
\begin{gather*}
\zeta_p = - \lambda_{p,1} - \sum_{t \in D_0} \lambda_{t,1} \quad (p \in \sfQ_\infty^v), \\
\zeta_{[t,l]} = \lambda_{t,l} - \lambda_{t,l+1}, \qquad
\zeta_{[p,l]} = \lambda_{p,l} - \lambda_{p,l+1},
\end{gather*}
and a collection $V=(V_i)_{i \in \sfQ^v}$ of vector spaces
as follows.
For $p \in \sfQ_\infty^v$, let $V_p$ be the one used above.
For $t \in D_0$ and $l=1,2,\dots ,d_t-1$, set
\[
V_{[t,l]} = \range \left( \prod_{i=1}^l 
\left( \pi_\sres (L_t) -\lambda_{t,i} 1_{\C^n} \right) \right),
\]
and also, for $p \in \sfQ_\infty^v$ and $l=1,2,\dots ,d_p-1$, set
\[
V_{[p,l]} = \range \left( \prod_{i=1}^l 
\left( \pi_\sres (L_\infty)|_{V_p} -\lambda_{p,i} 1_{V_p} \right) \right).
\]

Applying \lmmref{lem:quiver-A} to $O(L_t)$, $O_p(L_\infty)$, 
we then obtain an open embedding
\[
\varphi \colon \calM^*(\bfT,\bfL) \hookrightarrow \Rep_{\wh{\sfQ}}(V) \GIT_{\zeta_V} G_V.
\]
We will show that $\varphi$ maps the subset $\calM^*_s(\bfT,\bfL)$ of $\calM^*(\bfT,\bfL)$
consisting of all stable points in the following sense 
onto the quiver variety $\frM_\sfQ^s(V,\zeta)$.

\begin{dfn}\label{dfn:stable}
A meromorphic connection $A \in \frg \otimes_\C \Omega^1(*D)$ on 
the trivial vector bundle $\calO^{\oplus n}$ is {\em stable}
if it has no non-zero proper subspace $S \subset \C^n$
such that $A(S \otimes_\C \calO) \subset S \otimes_\C \Omega^1(*D)$.
\end{dfn}

\begin{thm}\label{thm:main}
There exists a symplectomorphism 
$\calM^*_s(\bfT,\bfL) \simeq \frM_\sfQ^s(V,\zeta)$.
\end{thm}

\begin{proof}
For $t \in D_0$ and $p \in \sfQ_\infty^v$, 
let $\alpha_{p,t}$ be the arrow from $[t,1]$ to $p$.
Then the image of $\varphi$ is exactly the set of $G_V$-orbits $[\Xi]$
in $\mu^{-1}(\zeta_V)$ such that 
\begin{itemize}
\item $\Ker \Xi_\alpha =0$ and $\range \Xi_{\ov{\alpha}} = V_{\sfs(\alpha)}$ 
for any $\alpha \in \sfQ_0^a \setminus \{\, \alpha_{p,t} \mid p \in \sfQ_\infty^v, t \in D_0 \,\}$, 
\item $\bigcap_{p \in \sfQ_\infty^v} \Ker \Xi_{\alpha_{p,t}} =0$ and 
$\sum_{p \in \sfQ_\infty^v} \range \Xi_{\ov{\alpha_{p,t}}} = \C^n$ for all $t \in D_0$.
\end{itemize}
We will check that the map $\varphi$ gives a symplectomorphism between the stable parts.
Let $A$ be a meromorphic connection representing a point in $\calM^*(\bfT,\bfL)$
and $[\Xi]=\varphi([A])$.
Take a standard coordinate $z$ on the affine line $\bbP^1 \setminus \{ \infty \}$.
Then $A$ is expressed as
\[
A = \left( \sum_{i=0}^{k-2} A_i z^i + \sum_{t \in D_0} \frac{R_t}{z-z(t)} \right) dz,
\quad A_i, R_t \in \frg.
\]
Set $B = \sum_{i=0}^{k-2} A_i z^i dz \in \frb_k^*$.
By the definition of $\varphi$, we may assume that $B$ is contained in 
the $B_k$-coadjoint orbit through $dT_\infty$ 
and each $R_t$ is expressed as 
\[
R_t = \left( \Xi_{\alpha_{p,t}} \Xi_{\ov{\alpha_{q,t}}} \right)_{p,q} + \lambda_{t,1}1_{\C^n}.
\]

Now suppose that $A$ is stable and let $W=(W_i)_{i \in \sfQ^v}$ be 
a $\Xi$-invariant subspace.
Then the above expression of $R_t$ and \lmmref{lem:stability} show that 
the direct sum $S := \bigoplus_{p \in \sfQ_\infty^v} W_p$ is invariant 
under all $A_i$ and $R_t$. 
Therefore the stability of $A$ implies that $S=0$ or $S=\C^n$.
If $S=0$, the injectivity conditions for $\Xi_\alpha$, $\alpha \in \sfQ_0^a$
immediately show that $W_i =0$ for all $i \in \sfQ^v$, and 
if $S=\C^n$, the surjectivity conditions for $\Xi_{\ov{\alpha}}$, $\alpha \in \sfQ_0^a$
show that $W_i = V_i$ for all $i \in \sfQ^v$. 
Hence $\Xi$ is stable.

Conversely, suppose that $\Xi$ is stable and 
let $S \subset \C^n$ be a subspace invariant under all $A_i$ and $R_t$.
Note that it is then also invariant under $\res_\infty(A)$.
\lmmref{lem:stability} shows 
$S = \sum_{p \in \sfQ_\infty^v} (S \cap V_p)$ 
and that the collection $(W_p)_{p \in \sfQ_\infty^v}$, $W_p := S \cap V_p$ 
is invariant under $(\Xi_\alpha)_{\alpha \in \sfQ_\infty^a}$.
Applying \lmmref{lem:quiver-A} to all $R_t$ and the block components 
$\res_\infty(A)|_{\gl(V_p)}$,
we thus obtain a $\Xi$-invariant subspace $W=(W_i)_{i \in \sfQ^v}$ 
containing $(W_p)_{p \in \sfQ_\infty^v}$ as a subcollection.
The stability of $\Xi$ then implies $W=0$ or $W=V$, and hence $S=0$ or $S=\C^n$.
\end{proof}

The above theorem enables us to apply 
Crawley-Boevey's criterion~\cite{CB01} 
for the non-emptiness of quiver varieties 
to the additive irregular Deligne-Simpson problem
and obtain the following result, 
generalizing \cite{Boa12,CB03DS}
(see \cite{Kac83} for the basic terminology on the root systems attached to quivers):

\begin{crl}\label{cor:DS}
The space $\calM^*_s(\bfT,\bfL) \simeq \frM_\sfQ^s(\bv,\zeta)$ 
is non-empty if and only if the following conditions hold:
\begin{enumerate}
\item $\bv \in \Z_{\geq 0}^{\sfQ^v}$ is a positive root,
\item $\zeta \cdot \bv =0$,
\item any non-trivial decomposition 
$\bv = \bw_1 + \bw_2 + \cdots + \bw_l$
of $\bv$ by positive roots $\bw_j$ with $\zeta \cdot \bw_j =0$ 
satisfies the inequality 
\[
\Delta(\bv) > \sum_{j=1}^l \Delta(\bw_j).
\] 
\end{enumerate}
\end{crl}

\section{Appendix}\label{sec:app}

\subsection{Formal reduction theory}\label{subsec:reduction}

In this subsection,
we recall some basic facts 
due to Babbitt-Varadarajan~\cite{BV}
on the formal reduction theory 
of meromorphic connections. 
Let $G$ be a complex reductive group and $\frg$ its Lie algebra.

\begin{prp}[{\cite[Proposition 9.3.2]{BV}}]\label{prop:BV1}
Let $A =\sum_{i \geq 0}A_i z^{i-k}\,dz \in \frg\fl{z}dz$ 
with $k>1$, $A_0 \neq 0$. 
Suppose that $A_0$ is semisimple.
Then there exists a unique
$\wh{g} \in G(\C\fp{z})$
of the form
\[
\wh{g}(z) = \prod_{i=1}^\infty \exp (z^i X_i) 
= \lim_{j \to \infty} \left( \exp (z^j X_j) \exp (z^{j-1} X_{j-1}) 
\cdots \exp (z X_1) \right)
\]
with
$X_i \in \range \ad_{A_0}\, (i \in \Z_{>0})$,
such that 
\[
A'=\sum_{i \geq 0}A'_i z^{i-k} dz :=\wh{g}[A]
\]
satisfies 
\[
A'_0 = A_0, \qquad [A_0,A']=0.
\]
If there exists $p \in \Z_{>0}$ such that 
$[A_0, A_i]=0$ for $i \leq p$, then $A'_i = A_i$ for $i \leq p$.
\end{prp}

\begin{prp}\label{prop:BV2}
Let $A =\sum_{i \geq 0}A_i z^{i-k}\,dz \in \frg\fl{z}dz$ 
with $k>1$, $A_0 \neq 0$. 
Suppose that $A_i,\,i \leq k-2$ are contained in some torus $\frt \subset \frg$. Set 
\[
H_i = \bigcap_{j=0}^{k-i-2} \{\, h \in G \mid \Ad_h A_j = A_j\,\} 
\quad (i=0,1, \dots , k-2),
\quad H_{k-1} = G,
\]
and let $\frh_i$ be the Lie algebra of $H_i$ for $i \leq k-1$.
Then there exists a unique $(k-1)$-tuple  
$(\wh{g}^{(1)}, \wh{g}^{(2)}, \dots , \wh{g}^{(k-1)})$ with
\[
\wh{g}^{(l)}(z) = \prod_{i=1}^\infty \exp (z^i X^{(l)}_i) \in H_{k-l}(\C\fp{z}), 
\quad 
X^{(l)}_i \in \range \left( \ad_{A_{l-1}}|_{\frh_{k-l}} \right)
\]
such that
\[
A^{(l)}=\sum_{i \geq 0}A^{(l)}_i z^{i-k} dz
:=\wh{g}^{(l)} \cdots \wh{g}^{(1)}[A]
\]
satisfies
\[
A^{(l)}_i = A_i \quad (i=0,1, \dots ,k-2),
\qquad A^{(l)} \in \frh_{k-l-1}\fl{z}dz
\]
for each $l \geq 1$.
\end{prp}

\begin{proof}
\prpref{prop:BV1} shows that there uniquely exists 
\[
\wh{g}^{(1)}(z)=\prod_{i>0} \exp (z^i X^{(1)}_i) \in G(\C\fp{z}),
\quad X^{(1)}_i \in \range \ad_{A_0}
\]
such that $A^{(1)}=\sum_{i\geq 0} A^{(1)}_i z^{i-k}dz := \wh{g}^{(1)}[A]$ satisfies 
\[
A^{(1)}_i = A_i \quad (i \leq k-2), \qquad
A^{(1)} \in \frh_{k-2}\fl{z}dz.
\] 
If $k=2$, we are done. 
Otherwise, we apply \prpref{prop:BV1} to 
$A^{(1)}-A_0 z^{-k}dz \in \frh_{k-2}\fl{z}dz$
with $G$ replaced by $H_{k-2}$,
and uniquely find 
\[
\wh{g}^{(2)}(z)=\prod_{i>0} \exp (z^i X^{(2)}_i) \in H_{k-2}(\C\fp{z}),
\quad X^{(2)}_i \in \range \left( \ad_{A_1} |_{\frh_{k-2}} \right)
\]
such that 
\[
A^{(2)}=\sum_{i\geq 0} A^{(2)}_i z^{i-k}dz := \wh{g}^{(2)}[A^{(1)}]
=\wh{g}^{(2)}[A^{(1)}-A_0 z^{-k}dz] + A_0 z^{-k}dz
\]
satisfies 
\[
A^{(2)}_i = A_i \quad (i \leq k-2), \qquad
A^{(2)} \in \frh_{k-3}\fl{z}dz.
\]
Repeating this argument yields the assertion.
\end{proof}

\begin{crl}\label{cor:reduction}
Under the assumption and notation of \prpref{prop:BV2},
there exists $\wh{g} \in G(\C\fp{z})$ with $\wh{g}(0)=1$ such that 
\[
A'=\sum_{i \geq 0}A'_i z^{i-k} dz :=\wh{g}[A]
\]
satisfies 
\[
A'_i = A_i \quad (i=0,1, \dots ,k-2),
\qquad A' \in \frh_0\fl{z}dz.
\]
\end{crl}

\subsection{Coadjoint orbits of general linear groups and quivers of type $A$}\label{subsec:quiver-A}

In this subsection, we assume $G=\GL(n,\C)$ 
and recall some relation 
between $G$-coadjoint orbits and quivers of type $A$.

Let $O \subset \frg^*$ be a $G$-coadjoint orbit.
We identify $\frg^*$ with $\frg$ via the trace pairing.

\begin{dfn}\label{dfn:marking}
A {\em marking} of $O$ is an ordered tuple 
$(\lambda_1,\lambda_2,\dots ,\lambda_d)$
of complex numbers such that
\[
\prod_{i=1}^d (A-\lambda_i 1_{\C^n}) = 0
\]
for some (and hence all) $A \in O$.
\end{dfn}

Fix a marking 
$(\lambda_1,\lambda_2,\dots ,\lambda_d)$ of $O$
and define a quiver $\sfQ$ with vertices
$\sfQ^v = \{\, 0,1, \dots ,d-1\,\}$
by drawing one arrow from each $l \in \sfQ^v\, (l \geq 1)$ to $l-1$.
Fix $L \in O$ and define a collection of vector spaces $V=(V_l)_{l \in \sfQ^v}$ by
\[
V_0 = \C^n, \qquad
V_l = \range \left( \prod_{i=1}^l 
\left( L -\lambda_i 1_{\C^n} \right) \right)
\quad (l \geq 1).
\]
For $\Xi \in \Rep_{\wh{\sfQ}}(V)$, 
we denote its components by $\Xi_{l,l-1} \in \Hom(V_{l-1},V_l)$, 
$\Xi_{l-1,l} \in \Hom(V_l,V_{l-1})$, $l=1,2, \dots ,d-1$.
Note that in this case the moment map $\mu=(\mu_l) \colon \Rep_{\wh{\sfQ}}(V) \to \frg_V$ 
is expressed as
\[
\mu_l(\Xi) =
\begin{cases}
\Xi_{0,1}\Xi_{1,0} & (l=0), \\
\Xi_{l,l+1}\Xi_{l+1,l}-\Xi_{l,l-1}\Xi_{l-1,l} & (1 \leq l < d-1), \\
-\Xi_{d-1,d-2}\Xi_{d-2,d-1} & (l=d-1).
\end{cases}
\]

The following lemma is essentially due to Crawley-Boevey~\cite{CB03,CB03DS}:

\begin{lmm}\label{lem:quiver-A}
Let $Z$ be the subvariety of $\Rep_{\wh{\sfQ}}(V)$ defined by 
\[
Z = \set{\Xi \in \Rep_{\wh{\sfQ}}(V)}%
{\begin{aligned}
&\mu_l(\Xi)=(\lambda_l-\lambda_{l+1})1_{V_l}\ (l \geq 1),\\
&\text{$\Xi_{l-1,l}$ is injective for $l \geq 1$}, \\
&\text{$\Xi_{l,l-1}$ is surjective for $l \geq 1$}
\end{aligned}}.
\]
Then it is smooth, the group $\prod_{l=1}^{d-1} \GL(V_l)$ acts freely there,
and the shift of the moment map $\mu_0$ by $\lambda_1 1_{\C^n}$ 
induces a $G$-equivariant symplectomorphism
\[
Z/\prod_{l=1}^{d-1}\GL(V_l) \to O; \quad [\Xi] \mapsto \Xi_{0,1}\Xi_{1,0} + \lambda_1 1_{\C^n}.
\]
Furthermore, for any $\Xi \in Z$ and any subspace $S$ of $\C^n$ invariant under  
$\Xi_{0,1}\Xi_{1,0} + \lambda_1 1_{\C^n}$, 
there exists a $\Xi$-invariant subspace $W=(W_l)_{l \in \sfQ^v}$ of $V$ 
such that $W=0$ (resp.\ $W=V$) if and only if $S=0$ (resp.\ $S=\C^n$).
\end{lmm}

\begin{proof}
It is straightforward to check that the action of $\prod_{l=1}^{d-1}\GL(V_l)$
on $Z$ is free and hence $Z$ is smooth (by a basic property of the moment map).
Also, it is known~\cite[Appendix]{CB03} that the map $\mu_0 + \lambda_1 1_{\C^n}$ induces 
an isomorphism from the categorical quotient of 
the affine algebraic variety
\[
\ov{Z} := \{\, \Xi \in \Rep_{\wh{\sfQ}}(V) 
\mid \mu_l(\Xi)=(\lambda_l - \lambda_{l+1})1_{\C^n}\ (l \geq 1)\,\}
\]
by the action of $\prod_{l=1}^{d-1}\GL(V_l)$
to the closure of $O$.
The arguments in the proof of \cite[Theorem~1]{CB03DS} then shows that  
it induces a $G$-equivariant isomorphism 
from $Z/\prod_{l=1}^{d-1}\GL(V_l)$ to $O$,
which is Poisson (and hence symplectic) because it comes from a moment map.
The proof of the rest assertion is also included in [loc.~cit.]:
Let $\Xi \in Z$ and $X=\Xi_{0,1}\Xi_{1,0} + \lambda_1 1_{\C^n} \in O$.
If $S \subset \C^n$ is $X$-invariant,
define 
\[
W_0 =S, \qquad W_l = \Xi_{l,l-1} \Xi_{l-1,l-2} \cdots \Xi_{1,0}(S)
\quad (l \geq 1).
\] 
Then $W=(W_l)_{l \in \sfQ^v}$ satisfies the desired condition.
\end{proof}



\providecommand{\bysame}{\leavevmode\hbox to3em{\hrulefill}\thinspace}



\begin{thebibliography}{99}

\bibitem{Ari10}
D.~Arinkin, \emph{Rigid irregular connections on {$\Bbb P\sp 1$}}, Compos.
  Math. \textbf{146} (2010), no.~5, 1323--1338.

\bibitem{BV}
D.~G. Babbitt and V.~S. Varadarajan, \emph{Formal reduction theory of
  meromorphic differential equations: a group theoretic view}, Pacific J. Math.
  \textbf{109} (1983), no.~1, 1--80.

\bibitem{Boa01}
P.~Boalch, \emph{Symplectic manifolds and isomonodromic deformations}, Adv.
  Math. \textbf{163} (2001), no.~2, 137--205.

\bibitem{Boa08}
\bysame, \emph{Irregular connections and {K}ac-{M}oody root systems},  (2008).

\bibitem{Boa09}
\bysame, \emph{Quivers and difference {P}ainlev\'e equations}, Groups and
  symmetries, CRM Proc. Lecture Notes, vol.~47, Amer. Math. Soc., Providence,
  RI, 2009, pp.~25--51.

\bibitem{Boa12}
\bysame, \emph{Simply-laced isomonodromy systems}, Publ. Math. Inst. Hautes
  \'Etudes Sci. \textbf{116} (2012), no.~1, 1--68.

\bibitem{CB01}
W.~Crawley-Boevey, \emph{Geometry of the moment map for representations of
  quivers}, Compositio Math. \textbf{126} (2001), no.~3, 257--293.

\bibitem{CB03}
\bysame, \emph{Normality of {M}arsden-{W}einstein reductions for
  representations of quivers}, Math. Ann. \textbf{325} (2003), no.~1, 55--79.

\bibitem{CB03DS}
\bysame, \emph{On matrices in prescribed conjugacy classes with no common
  invariant subspace and sum zero}, Duke Math. J. \textbf{118} (2003), no.~2,
  339--352.

\bibitem{DR00}
M.~Dettweiler and S.~Reiter, \emph{An algorithm of {K}atz and its application
  to the inverse {G}alois problem}, J. Symbolic Comput. \textbf{30} (2000),
  no.~6, 761--798, Algorithmic methods in Galois theory.

\bibitem{Hir}
K.~Hiroe, \emph{Linear differential equations on the Riemann sphere
and the representations of quivers}, preprint, arXiv:1307.7438.

\bibitem{JMU}
M.~Jimbo, T.~Miwa, and K.~Ueno, \emph{Monodromy preserving deformation of
  linear ordinary differential equations with rational coefficients. {I}.
  {G}eneral theory and {$\tau $}-function}, Phys. D \textbf{2} (1981), no.~2,
  306--352.

\bibitem{Kac83}
V.~G. Kac, \emph{Root systems, representations of quivers and invariant
  theory}, Invariant theory ({M}ontecatini, 1982), Lecture Notes in Math., vol.
  996, Springer, Berlin, 1983, pp.~74--108.

\bibitem{Katz}
N.~M. Katz, \emph{Rigid local systems}, Annals of Mathematics Studies, vol.
  139, Princeton University Press, Princeton, NJ, 1996.

\bibitem{King}
A.~D. King, \emph{Moduli of representations of finite-dimensional algebras},
  Quart. J. Math. Oxford Ser. (2) \textbf{45} (1994), no.~180, 515--530.

\bibitem{MFK}
D.~Mumford, J.~Fogarty, and F.~Kirwan, \emph{Geometric invariant theory}, third
  ed., Ergebnisse der Mathematik und ihrer Grenzgebiete (2) [Results in
  Mathematics and Related Areas (2)], vol.~34, Springer-Verlag, Berlin, 1994.

\bibitem{Nak94}
H.~Nakajima, \emph{Instantons on {ALE} spaces, quiver varieties, and
  {K}ac-{M}oody algebras}, Duke Math. J. \textbf{76} (1994), no.~2, 365--416.

\bibitem{Nak98}
\bysame, \emph{Quiver varieties and {K}ac-{M}oody algebras}, Duke Math. J.
  \textbf{91} (1998), no.~3, 515--560.

\bibitem{Nak01AMS}
\bysame, \emph{Quiver varieties and finite-dimensional representations of
  quantum affine algebras}, J. Amer. Math. Soc. \textbf{14} (2001), no.~1,
  145--238.

\bibitem{Nak01}
\bysame, \emph{Quiver varieties and tensor products}, Invent. Math.
  \textbf{146} (2001), no.~2, 399--449.

\bibitem{Nak03}
\bysame, \emph{Reflection functors for quiver varieties and {W}eyl group
  actions}, Math. Ann. \textbf{327} (2003), no.~4, 671--721.

\bibitem{Nak04}
\bysame, \emph{Quiver varieties and {$t$}-analogs of {$q$}-characters of
  quantum affine algebras}, Ann. of Math. (2) \textbf{160} (2004), no.~3,
  1057--1097.

\bibitem{Ros}
M.~Rosenlicht, \emph{On quotient varieties and the affine embedding of certain
  homogeneous spaces}, Trans. Amer. Math. Soc. \textbf{101} (1961), 211--223.

\bibitem{Yam08}
D.~Yamakawa, \emph{Geometry of multiplicative preprojective algebra}, Int.
  Math. Res. Pap. IMRP (2008), Art. ID rpn008, 77pp.

\bibitem{Yam10}
\bysame, \emph{Quiver varieties with multiplicities, {W}eyl groups of
  non-symmetric {K}ac-{M}oody algebras, and {P}ainlev\'e equations}, SIGMA
  Symmetry Integrability Geom. Methods Appl. \textbf{6} (2010), Paper 087, 43.

\bibitem{Yam11}
\bysame, \emph{Middle convolution and {H}arnad duality}, Math. Ann.
  \textbf{349} (2011), no.~1, 215--262.

\end{thebibliography}
\end{document}